\documentclass[10pt]{article}
\font\smallit=cmti10

\usepackage{amssymb,latexsym,amsmath,epsfig,amsthm,url,mathrsfs,tikz,tkz-graph,caption,mathrsfs} 

\usetikzlibrary{arrows}
\usetikzlibrary{shapes}
\usetikzlibrary{automata}

\makeatletter

\renewcommand\section{\@startsection {section}{1}{\z@}
{-30pt \@plus -1ex \@minus -.2ex}
{2.3ex \@plus.2ex}
{\normalfont\normalsize\bfseries\boldmath}}

\renewcommand\subsection{\@startsection{subsection}{2}{\z@}
{-3.25ex\@plus -1ex \@minus -.2ex}
{1.5ex \@plus .2ex}
{\normalfont\normalsize\bfseries\boldmath}}

\renewcommand{\@seccntformat}[1]{\csname the#1\endcsname. }

\makeatother

\newtheorem{theorem}{Theorem}

\newtheorem{corollary}{Corollary}

\theoremstyle{definition}
\newtheorem{definition}{Definition}

\newtheorem{remark}{Remark}
\newtheorem{example}{Example}

\begin{document}

\begin{center}
\uppercase{\bf Finding Permutiples of a Known Base and Multiplier}
\vskip 20pt
{\bf Benjamin V. Holt}\\
{\smallit Department of Mathematics, Southwestern Oregon Community College, Oregon}\\
{\tt benjamin.holt@socc.edu}\\
\vskip 10pt
\end{center}
\vskip 20pt

\centerline{\bf Abstract}

\noindent
Natural numbers which are nontrivial multiples of some permutation of their base-$b$ digit representations are called permutiples. Specific cases include numbers which are multiples of cyclic permutations (cyclic numbers) and reversals of their digits (palintiples). Previous efforts have produced methods which construct new examples of permutiples with the same set of digits as a known example. Using simple graph-theoretical and finite-state machine constructions, we advance previous work by describing two methods for finding permutiples of a known base and multiplier with no need for known examples or prior knowledge of digits. 

\pagestyle{myheadings}
\thispagestyle{empty}
\baselineskip=12.875pt
\vskip 30pt

\section{Introduction}

A {\it permutiple} is the result of a ``digit-preserving'' multiplication. A bit more formally, a permutiple is a natural number whose representation in some base is an integer multiple of some permutation of its digits \cite{holt_3}. Equivalently, a permutiple can be described as a number which is divisible by a permutation of its digits.

The most well-known cases of permutiples are palintiple numbers \cite{holt_1,holt_2} and cyclic numbers \cite{guttman,kalman}. Palintiples, also known as reverse multiples \cite{kendrick_1,sloane, young_1,young_2}, are multiples of their digit reversals. Two base-10 examples include $87912=4\cdot 21978$ and $98901 = 9 \cdot 10989.$  Cyclic numbers are multiples of cyclic permutations of their digits, such as $714285 = 5 \cdot 142857.$ Cyclic numbers are relatively easy to describe when compared to palintiple numbers, which have a great deal more variety and are much more difficult to characterize. Using the work of Young \cite{young_1,young_2}, Sloane \cite{sloane} developed a graph-theoretical construction, based upon the interactions between the digits and carries, called {\it Young graphs}. Young graphs describe palintiples of any length with a known base and multiplier, and they offer a way to classify palintiples by graph isomorphism. A digit-specific collection of permutiples studied by Qu and Curran \cite{qu} are base-$b$ numbers which are multiples of $(b^{b-1}-1)/(b-1)^2.$ Base-10 examples include $987654312=8 \cdot 123456789$ and $493827156=4 \cdot 123456789.$ 

Using elementary techniques, the author \cite{holt_3, holt_4} develops methods for finding new examples of general permutiples from known examples. For instance, these methods produce new examples, such as $79128=4\cdot 19782,$ from the palintiple example above with the same digits. The results of \cite{holt_4} improve upon the work of \cite{holt_3} by describing a method for finding all permutiples with the same set of digits, multiplier, and length from a single known example.  

This effort develops two methods for finding permutiples of a known base and multiplier. The first method advances the work of \cite{holt_4} by finding all permutiples of a given base, multiplier, and length, but without the requirement of having a known example in advance. All of the above is accomplished by using a basic graph-theoretical construction which will enable us to classify permutiples much more broadly than we could with the notion of permutiple conjugacy, introduced in \cite{holt_4}. Using a modification of the methods of Hoey \cite{hoey} and Sloane \cite{sloane}, we describe a finite-state machine which recognizes a language containing string representations of permutiples. From there, we develop a method for finding permutiples of any acceptable length for a base and multiplier of our choosing. 

\section{Basic Notation, Definitions, and Results}	

For the reader's convenience, we summarize and state the basic definitions and results from \cite{holt_3,holt_4}, which we will be using extensively in the sections which follow. Using $(d_k,d_{k-1},\ldots,d_0)_b$ to denote the 
natural number $\sum_{j=0}^{k}d_j b^j,$ where $0\leq d_j<b$ for all $0\leq j \leq k,$ the following is our definition of a permutiple number.

\begin{definition}[\cite{holt_3}]
Let $1<n<b$ be a natural number and $\sigma$ be a permutation on $\{0,1,2,\ldots, k\}$.
We say that $(d_k, d_{k-1},\ldots, d_0)_b$  is an $(n,b,\sigma)$-\textit{permutiple} provided 
\[
(d_k,d_{k-1},\ldots,d_1, d_0)_b=n(d_{\sigma(k)},d_{\sigma(k-1)},\ldots, d_{\sigma(1)}, d_{\sigma(0)})_b.
\]
When the permutation, $\sigma,$ is not pertinent to the discussion, we may also simply refer to $(d_k,d_{k-1},\ldots,d_0)_b$ as an $(n,b)$-permutiple. We shall refer to the collection of all base-$b$ permutiples having multiplier $n$ as {\it $(n,b)$-permutiples}.
\end{definition}

Coupling this definition with a description of single-digit multiplication gives us the next result.

\begin{theorem}[\cite{holt_3}] 
Let $(d_k, d_{k-1},\ldots, d_0)_b$ be an $(n,b,\sigma)$-permutiple, and let $c_j$ be the $j$th carry. Then,
\[
b c_{j+1}-c_j=nd_{\sigma(j)}-d_{j}
\]
for all $0\leq j \leq k$. 
\label{digits_carries_1}
\end{theorem}

The next result will also be central to the developments in the following sections.

\begin{theorem}[\cite{holt_3}] \label{carries}
Let $(d_k, d_{k-1},\ldots, d_0)_b$ be an $(n,b,\sigma)$-permutiple, and let $c_j$ be the
$j$th carry. Then, $c_j\leq n-1$ for all $0 \leq j \leq k$.
\end{theorem}

If we suppose that $(d_k,d_{k-1},\ldots, d_0)_b$ is an $(n,b,\sigma)$-permutiple, the central problem addressed by \cite{holt_3,holt_4} is to find all permutations, $\pi,$ such that the number given by $(d_{\pi(k)},d_{\pi(k-1)},\ldots, d_{\pi(0)})_b$ is also a permutiple. The notion of permutiple {\it conjugacy} helps us to find and classify these new examples.

\begin{definition}[\cite{holt_3}] \label{conj_class_def}
 Suppose $(d_k, d_{k-1},\ldots, d_0)_b$ is an $(n,b,\sigma)$-permutiple. Then, an $(n,b, \tau_1)$-permutiple, $(d_{\pi_1(k)}, d_{\pi_1(k-1)},\ldots, d_{\pi_1(0)})_b,$ and
 an $(n,b, \tau_2)$-permutiple, $(d_{\pi_2(k)}, d_{\pi_2(k-1)},\ldots, d_{\pi_2(0)})_b$,
  are said to be \textit{conjugate} if $\pi_1 \tau_1 \pi_1^{-1}=\pi_2 \tau_2 \pi_2^{-1}$.
\end{definition}
We note that this definition requires that we treat repeated digits as being distinct from one another. That is, we treat the collection of digits, $\{d_k, d_{k-1},\ldots, d_0\},$ as a multiset.

To make the above definition more concrete, we present the results of the first example of \cite{holt_4}, which uses the digits of a known example, $p=(d_4,d_3,d_2,d_1,d_0)_{10}=(8,7,9,1,2)_{10}=4 \cdot (2,1,9,7,8)_{10},$ to produce all $(4,10,\tau)$-permutiple examples, 
\[
\begin{array}{rl}
(\hat{d}_4,\hat{d}_3,\hat{d}_2,\hat{d}_1,\hat{d}_0)_{10}&=(d_{\pi(4)},d_{\pi(3)},d_{\pi(2)},d_{\pi(1)},d_{\pi(0)})_{10}\\\\
&=4\cdot (\hat{d}_{\tau(4)},\hat{d}_{\tau(3)},\hat{d}_{\tau(2)},\hat{d}_{\tau(1)},\hat{d}_{\tau(0)})_{10},
\end{array}
\]
 in the conjugacy class containing $p.$ Letting $\psi$ be the 5-cycle $(0,1,2,3,4),$ $\rho$ the reversal permutation, and $\varepsilon$ the identity permutation, the conjugacy class of $p$ is given in Table \ref{conj_class_table}.
\begin{center}
\begin{footnotesize}
 \begin{tabular}{|c|c|c|}
\hline $(4,10,\tau)$-Example & $\pi$ & $\tau$ \\\hline
$(8,7,9,1,2)_{10}=4 \cdot (2,1,9,7,8)_{10}$ & $\varepsilon$ & $\rho$  \\\hline
$(8,7,1,9,2)_{10}=4 \cdot (2,1,7,9,8)_{10}$ & $(1,2)$ & $(1,2)\rho(1,2)$ \\\hline
$(7,9,1,2,8)_{10}=4 \cdot (1,9,7,8,2)_{10}$ & $\psi^4$ & $\psi^{-4}\rho\psi^4$ \\\hline
$(7,1,9,2,8)_{10}=4 \cdot (1,7,9,8,2)_{10}$ & $(1,2)\psi^4$ & $\psi^{-4}(1,2)\rho(1,2)\psi^4$ \\\hline
\end{tabular}
\end{footnotesize}
\captionof{table}{The conjugacy class of $p=(8,7,9,1,2)_{10}=4 \cdot (2,1,9,7,8)_{10}.$}\label{conj_class_table}
\end{center}
Clearly, for every $\pi$ and $\tau$ in the above collection, we have that $\pi\tau\pi^{-1}=\rho.$ That is, all of these examples are conjugate permutiples. For details regarding how the above permutiples were obtained, the reader is directed to Example 1 of \cite{holt_4}.

\section{A Graph-Theoretical Approach}

To motivate our next definition, we shall consider the permutiples in the conjugacy class given in Table \ref{conj_class_table}. We notice that if we define a directed graph with all base-10 digits as vertices and $\left(d_j,d_{\tau(j)}\right)$ as edges,  the result is the same for all of these examples, which is shown in Figure \ref{conj_class_graph}.
\begin{center}
\begin{tikzpicture}
\tikzset{edge/.style = {->,> = latex'}}
\tikzset{vertex/.style = {shape=circle,draw,minimum size=1.5em}}
[xscale=3, yscale=3, auto=left,every node/.style={circle,fill=blue!20}]
\node[vertex] (n0) at (12,5) {$0$};
\node[vertex] (n1) at (11.618,6.17557) {$1$};
\node[vertex] (n2) at (10.618,6.90211) {$2$};
\node[vertex] (n3) at (9.38197,6.90211) {$3$};
\node[vertex] (n4) at (8.38197,6.17557) {$4$};
\node[vertex] (n5) at (8,5.00001) {$5$};
\node[vertex] (n6) at (8.38196,3.82443) {$6$};
\node[vertex] (n7) at (9.38196,3.09789) {$7$};
\node[vertex] (n8) at (10.618,3.09788) {$8$};
\node[vertex] (n9) at (11.618,3.82442) {$9$};
\draw[edge, bend right=10] (n8) to (n2);
\draw[edge, bend right=10] (n2) to (n8);
\draw[edge, bend right=10] (n7) to (n1);
\draw[edge, bend right=10] (n1) to (n7);
\draw[edge] (n9) to[loop] (n9);
\end{tikzpicture}
\captionof{figure}{The directed graph which results from taking the collection of ordered pairs $\left\{\left (d_j,d_{\sigma(j)}\right) |\, 0 \leq j \leq 4 \right\}$ from any example in Table \ref{conj_class_table} as directed edges.}
\label{conj_class_graph}
\end{center}
These observations give us an organizing principle which forms one of the bases of the methods of this effort.

\begin{definition}
Let $p=(d_k, d_{k-1},\ldots, d_0)_b$  be an $(n,b,\sigma)$-permutiple. We define a directed graph, called {\it the graph of $p$}, denoted as $G_p,$ to consist of the collection of base-$b$ digits as vertices, and the collection of directed edges $E_p=\left \{\left (d_j,d_{\sigma(j)}\right) |\, 0 \leq j \leq k \right \}.$  A graph, $G,$ for which there is a permutiple, $p,$ such that $G=G_p$ is called a {\it permutiple graph}.
\label{class_graph}
\end{definition}

At this point, we inform the reader that all graphs considered here will be directed graphs, and so, for the remainder of this effort, we shall refer to a ``directed graph'' as simply a ``graph.''

We now develop a condition which narrows down the possible collection of edges of a permutiple graph. Consider a generic $(n,b,\sigma)$-permutiple, $(d_k, d_{k-1},\ldots, d_0)_b.$ By Theorem \ref{digits_carries_1}, we have that $d_j+(b-n)d_{\sigma(j)} \equiv c_j \pmod b$ for all $0\leq j\leq k,$ where $c_j$ is the $j$th carry when multiplying by $n.$ Since $c_j\leq n-1$ by Theorem \ref{carries}, we have that $\lambda(d_j+(b-n)d_{\sigma(j)})\leq n-1$ for all $0\leq j\leq k,$ where $\lambda$ gives the least non-negative residue modulo $b.$ We state this as a result in the present graph-theoretical context.
\begin{theorem}
 Let $p=(d_k, d_{k-1},\ldots, d_0)_b$ be an $(n,b,\sigma)$-permutiple with graph $G_p.$ Then, for any edge, $(d_j,d_{\sigma(j)}),$ of $G_p,$ it must be that $\lambda \left(d_j+(b-n)d_{\sigma(j)}\right)\leq n-1$
for all $0\leq j\leq k,$ where $\lambda$ gives the least non-negative residue modulo $b.$
\end{theorem}
The above gives a list of candidates for the edges of a permutiple graph. Additionally, we note that the above result also follows from Theorem 3 of \cite{holt_4}, taking $\tau=\sigma$ and $\pi$ to be the identity permutation. These considerations motivate our next definition. 

\begin{definition}\label{mother_graph}
The $(n,b)$-\textit{mother graph}, denoted $M,$ is the graph having all base-$b$ digits as its vertices and the collection of edges, $(d_1,d_2),$ satisfying the inequality $\lambda \left(d_1+(b-n)d_{2}\right)\leq n-1.$ 
\end{definition}

Clearly, the graph of any $(n,b)$-permutiple is a subgraph of $M.$ We also notice that for a fixed $d_2,$ there is exactly one possibility for $d_1$ which satisfies the equation
\begin{equation}\label{input_pair}
\lambda \left(d_1+(b-n)d_{2}\right)=c_j.  
\end{equation}
In other words, there are $b$ ordered-pair solutions, $(d_1,d_2),$ to Equation (\ref{input_pair}). Since $0\leq c_j \leq n-1,$ it follows that the inequality in Definition \ref{mother_graph} has $nb$ ordered-pair solutions. This is to say that $M$ has precisely $nb$ edges. The above also shows that the indegree of each vertex is $n.$ Finally, since Equation (\ref{input_pair}) has $b$ solutions, we have that for a fixed value, $d_1,$ there is exactly one value of $d_2$ which solves Equation (\ref{input_pair}). Hence, for each $d_1,$ there are $n$ values of $d_2$ which solve the inequality in Definition \ref{mother_graph}. This shows that the outdegree of each vertex is also $n.$

As might be expected from our initial motivating example, permutiple conjugacy is related to the above ideas; if two $(n,b)$-permutiples, $p_1$ and $p_2,$ are conjugate, then their corresponding graphs are identical. We state this formally as a theorem.

 \begin{theorem}\label{conj_class}
 Suppose $(d_k, d_{k-1},\ldots, d_0)_b$ is an $(n,b,\sigma)$-permutiple.
  If an $(n,b, \tau_1)$-permutiple, $p_1=(d_{\pi_1(k)}, d_{\pi_1(k-1)},\ldots, d_{\pi_1(0)})_b,$ and
 an $(n,b, \tau_2)$-permutiple, $p_2=(d_{\pi_2(k)}, d_{\pi_2(k-1)},\ldots, d_{\pi_2(0)})_b,$ are in the same conjugacy class, then $G_{p_1}=G_{p_2}.$ 
 \end{theorem}
 
\begin{proof}
It is sufficient to show that $E_{p_1}=E_{p_2}.$ So,
\[
\begin{array}{lll}
 E_{p_1}&= \left\{\left(d_{\pi_1(j)},d_{\pi_1 \tau_1(j)}\right)\,|\, 0 \leq j \leq k\right\}&\\\\
 &= \left\{\left(d_j,d_{\pi_1 \tau_1 \pi_1^{-1}(j)}\right)\,|\, 0 \leq j \leq k\right\}&\\\\
&= \left\{\left(d_j,d_{\pi_2 \tau_2 \pi_2^{-1}(j)}\right)\,|\, 0 \leq j \leq k\right\}&\mbox{(by conjugacy)}\\\\
&= \left\{\left(d_{\pi_2(j)},d_{\pi_2 \tau_2(j)}\right)\,|\, 0 \leq j \leq k\right\}&\\\\
&=E_{p_2},&\\
\end{array}
\]
and the proof is complete.
\end{proof}

One of the limitations of classifying permutiples according to conjugacy, as it is defined above, is that only examples of a specified length can be conjugates. Another drawback, as mentioned earlier, is that repeated digits must be considered distinct from one another in order to make the classification. Permutiple graphs offer a way to sort and classify examples according to their most essential properties in a way that is independent of their length or presence of repeated digits.

We now make clear how we intend to handle permutations involving repeated digits. For any $(n,b,\sigma)$-permutiple, $p=(d_k,\ldots,d_0)_b,$ with repeated digits, $\sigma$ may not be the only permutation on $\{0,1,\ldots,k\}$ which gives the representation $(d_{\sigma(k)},\ldots,d_{\sigma(0)})_b.$
Taking the example
\[
(d_3,d_2,d_1,d_0)=(3,3,1,2)_4 = 2 \cdot (1,3,2,3)_4=2\cdot (d_{\sigma(3)},d_{\sigma(2)},d_{\sigma(1)},d_{\sigma(0)})_4
\]
as a case in point, we may take $\sigma$ to be either 
$\left(\begin{array}{cccc}0 & 1 & 2 & 3 \\ 3 & 0 & 2 & 1\end{array}\right)$ or $\left(\begin{array}{cccc}0 & 1 & 2 & 3 \\ 2 & 0 & 3 & 1\end{array}\right).$ It is clear by mapping the indices to their corresponding digits, using the same tableau presentation as above, we obtain the same result,
\[
\left(\begin{array}{cccc}d_0 & d_1 & d_2 & d_3 \\ d_3 & d_0 & d_2 & d_1\end{array}\right)
=\left(\begin{array}{cccc}d_0 & d_1 & d_2 & d_3 \\ d_2 & d_0 & d_3 & d_1\end{array}\right)
=\left(\begin{array}{cccc}2 & 1 & 3 & 3 \\ 3 & 2 & 3 & 1\end{array}\right),
\]
which is simply another representation of the collection of edges of this example's graph, $\{(2,3),(1,2),(3,3),(3,1)\}.$ 

The above example helps to underscore two elementary facts which hold in general: 1) any permutiple, $p,$ with repeated digits will have more than one permutation which describes the rearrangement of its digits, and 2) the graph of $p$ is unaffected by our choice from this collection of permutations. That said, some of these permutations are more advantageous than others in the sense that a permutiple's graph can convey information about the digit permutation itself provided that the permutation has a certain property. We shall presently state this property, noting that we  use the notation $(d_{j_1},d_{j_2},\ldots,d_{j_m-1},d_{j_m})$ to mean the collection of edges $\left \{(d_{j_1},d_{j_2}),(d_{j_2},d_{j_3}),\ldots (d_{j_m-1},d_{j_m}),(d_{j_m},d_{j_1})\right \},$ regardless of whether the collection is a cycle or a circuit on $G_p.$ We shall also refer to loops as 1-cycles.

\begin{definition}
Let $p=(d_k,\ldots,d_0)_b$ be an $(n,b,\sigma)$-permutiple with graph $G_p.$ Also, suppose $0\leq j \leq k,$ and let $m$ be the smallest natural number such that $j=\sigma^m(j).$ We say that $\sigma$ {\it preserves} the cycle $(j,\sigma(j),\ldots,\sigma^{m-1}(j))$ if $(d_{j},d_{\sigma(j)},\ldots,d_{\sigma^{m-1}(j)})$ is a cycle on $G_p.$ If $\sigma$ preserves every one of its disjoint cycles, then we say that $\sigma$ is {\it cycle-preserving}. 
\end{definition}

For the above example, $p=(d_3,d_2,d_1,d_0)=(3,3,1,2)_4 = 2 \cdot (1,3,2,3)_4,$ we see that 
$\left(\begin{array}{cccc}0 & 1 & 2 & 3 \\ 3 & 0 & 2 & 1\end{array}\right)=(0,3,1)(2)$ 
is cycle-preserving since $(d_0,d_3,d_1)(d_2)=(2,3,1)(3)$ is a union of cycles of $G_p.$ On the other hand,  
$\left(\begin{array}{cccc}0 & 1 & 2 & 3 \\ 2 & 0 & 3 & 1\end{array}\right)=(0,2,3,1)$ is not cycle-preserving since $(d_0,d_2,d_3,d_1)=(2,3,3,1)$ represents a circuit of $G_p,$ but is not a cycle.

\begin{theorem}
 For every $(n,b,\sigma)$-permutiple, $p=(d_k,\ldots,d_0)_b,$ there is a cycle-preserving permutation, $\hat{\sigma},$ such that
 $(d_k,\ldots, d_0)_b=n(d_{\hat{\sigma}(k)},\ldots, d_{\hat{\sigma}(0)})_b.$
\end{theorem}

\begin{proof}
If $\sigma$ is already cycle-preserving, then we are done. So, suppose that $\sigma$ is not cycle-preserving. Then, we can find a disjoint cycle, $\tau=(j,\sigma(j),\ldots,\sigma^{m-1}(j)),$ of $\sigma$ such that $(d_{j},d_{\sigma(j)},\ldots,d_{\sigma^{m-1}(j)})$ is not a cycle of $G_p,$ but a circuit. Therefore, there are at least two indices, $\sigma^{m_1}(j)$ and $\sigma^{m_2}(j),$ for which $d_{\sigma^{m_1}(j)}=d_{\sigma^{m_2}(j)}.$ Rewriting $\tau$ as 
\[
\left(j,\sigma(j),\ldots,\sigma^{m_1-1}(j),\sigma^{m_1}(j),\ldots,\sigma^{m_2-1}(j),\sigma^{m_2}(j),\ldots,\sigma^{m-1}(j)\right),
\] 
we construct a new permutation, 
\[
\hat{\tau}=\left(j,\sigma(j),\ldots,\sigma^{m_1-1}(j),\sigma^{m_2}(j),\ldots,\sigma^{m-1}(j)\right)\left(\sigma^{m_1}(j),\ldots,\sigma^{m_2-1}(j)\right).
\]
Both permutations, $\tau$ and $\hat{\tau},$ define the same collection of edges of $G_p.$ So, replacing $\tau$ with $\hat{\tau}$ in the cycle decomposition of $\sigma$ does not change the sequence of permuted digits. 

Now, if $\tau$ contains indices for which there are multiple digits with repeats, the above process can still be applied for each repeat until no disjoint cycle contains indices corresponding to repeated digits. Applying this process to every disjoint cycle of $\sigma,$ and naming the resulting product of cycles $\hat{\sigma},$ we see that $\hat{\sigma}$ preserves all of its disjoint cycles. Since $\hat{\sigma}$ defines the same collection of edges of $G_p,$ we have that $(d_k,\ldots, d_0)_b=n(d_{\sigma(k)},\ldots, d_{\sigma(0)})_b=n(d_{\hat{\sigma}(k)},\ldots, d_{\hat{\sigma}(0)})_b.$   
\end{proof}

\begin{remark}
 We note that the choice of cycle-preserving permutation is generally not unique. Also, it is clear that for any cycle-preserving permutation, $\sigma,$ we have for any $j$ such that $d_j=d_{\sigma(j)}$  that $\sigma(j)=j.$ That is, $j$ is a fixed point of $\sigma.$
\end{remark}

The above framework enables us to more broadly classify $(n,b)$-permutiples. 

\begin{definition}\label{perm_class}
Let $p$ be an $(n,b)$-permutiple with graph $G_p.$ We define the {\it class of $p$} to be the collection, $C,$ of all $(n,b)$-permutiples, $q,$ such that $G_q$ is a subgraph of $G_p.$ We also define the graph of the class to be $G_p,$ which we will denote as $G_C$ and will call the {\it graph of $C.$}
\end{definition}

The above definitions unify all of the $(4,10)$-permutiple examples considered here and in other works \cite{hoey,holt_1,holt_2,holt_3,holt_4,kendrick_1,sloane,young_1,young_2} into a single class. These include $(8,7,1,2)_{10}=4 \cdot (2,1,7,8)_{10},$ $(8,7,9,1,2)_{10}=4 \cdot (2,1,9,7,8)_{10},$ $(8,7,9,9,1,2)_{10}=4 \cdot (2,1,9,9,7,8)_{10},$ as well as $(7,1,2,8)_{10}=4 \cdot (1,7,8,2)_{10},$ $(7,9,1,2,8)_{10}=4 \cdot (1,9,7,8,2)_{10},$ $(7,9,9,1,2,8)_{10}=4 \cdot (1,9,9,7,8,2)_{10},$ $(8,7,1,2,8,7,1,2)_{10}=4 \cdot (2,1,7,8,2,1,7,8)_{10},$ and so forth.

\begin{remark}
Under the above definition, any two $(n,b)$-permutiples in the same conjugacy class are also members of the same permutiple class.
\end{remark}

\begin{theorem}
Let $C$ be any $(n,b)$-permutiple class. Then, $G_C$ is a union of cycles of $M.$
\label{cycle_union}
\end{theorem}

\begin{proof}
Since every edge of $G_C$ is also an edge of $M,$ it is sufficient to show that every vertex of $G_C$ having positive outdegree lies on some cycle. Choose an $(n,b,\sigma)$-permutiple, $p=(d_k, d_{k-1},\ldots, d_0)_b,$ so that $G_C=G_p$ and $\sigma$ is cycle-preserving. Then, it is sufficient to show that any digit, $d_j,$ of $p$ lies on a cycle of $G_p.$ Find the $m$-cycle, $(j_1,j_2,\ldots,j,\ldots, j_m),$ in the cycle decomposition of $\sigma$ which contains $j.$ Since $\sigma$ was chosen to be cycle-preserving, $(d_{j_1},$ $d_{j_2},$ $\ldots,$ $d_j,$ $\ldots,$ $d_{j_m})$ is a graph cycle which contains $d_j.$ 
\end{proof}

\section{A Method for Finding All Permutiples up to a Known Length}

Theorems \ref{digits_carries_1} and \ref{cycle_union} give us a method for finding all permutiples of a known base, $b,$ multiplier, $n,$ and length, $\ell,$ without the requirement of a specific example to start with. Theorem \ref{cycle_union} tells us that once we identify all of the digit cycles of the easily-constructed $(n,b)$-mother graph, $M,$ we may form combinations of these, with 1-cycles included, whose lengths add to $\ell.$ These, in turn, give candidate collections of possible permutiple digits. We may narrow down these candidates by using a corollary to Theorem \ref{digits_carries_1}; summing the equation in Theorem \ref{digits_carries_1} over all $0 \leq j \leq k=\ell-1$ gives us the equation
\begin{equation}\label{digits_carries_2}
(n-1)\sum_{j=0}^{k}d_j=(b-1)\sum_{j=1}^{k} c_j. 
\end{equation}
The results of \cite{holt_4} may then be applied to the remaining digit collections.

We now put the above into practice with an example.

\begin{example}
 Using the method described above, we shall find all $(2,6)$-permutiples up to length 5. The $(2,6)$-mother graph is seen in Figure \ref{2_6_mg}.
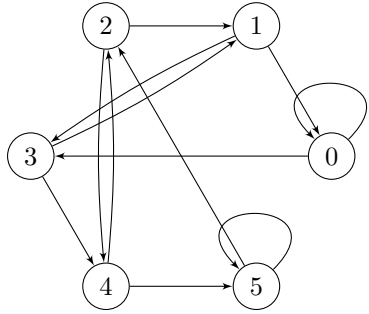
\begin{center}
\begin{tikzpicture} 
\tikzset{edge/.style = {->,> = latex'}} 
\tikzset{vertex/.style = {shape=circle,draw,minimum size=1.5em}} [xscale=2, yscale=2, auto=left,every node/.style={circle,fill=blue!20}] 
\node[vertex] (n0) at (12,5) {$0$}; 
\node[vertex] (n1) at (11.00000153205039,6.732049923038268) {$1$}; \node[vertex] (n2) at (9.000003064103128,6.732052576626029) {$2$}; \node[vertex] (n3) at (8.000000000007041,5.000005307179587) {$3$}; \node[vertex] (n4) at (8.999993871803133,3.2679527305616887) {$4$}; \node[vertex] (n5) at (10.999992339736313,3.2679447697984068) {$5$};
\draw[edge](n0) to[loop] (n0);
\draw[edge, bend right=0](n0) to (n3);
\draw[edge, bend right=0](n1) to (n0);
\draw[edge, bend right=5](n1) to (n3);
\draw[edge, bend right=0](n2) to (n1);
\draw[edge, bend right=5](n2) to (n4);
\draw[edge, bend right=5](n3) to (n1);
\draw[edge, bend right=0](n3) to (n4);
\draw[edge, bend right=5](n4) to (n2);
\draw[edge, bend right=0](n4) to (n5);
\draw[edge, bend right=0](n5) to (n2);
\draw[edge](n5) to[loop] (n5);
\end{tikzpicture}
\captionof{figure}{The $(2,6)$-mother graph.}
\label{2_6_mg}
\end{center}
The possible digit cycles and their lengths are presented in Table \ref{2_6_cycles}.
\begin{center}
 \begin{tabular}{|l|l|}
 \hline
1-cycles & $(0),$ $(5)$  \\\hline
2-cycles & $(1,3),$ $(2,4)$  \\\hline
3-cycles & $(0,3,1),$ $(2,4,5)$  \\\hline
4-cycles & $(1,3,4,2)$  \\\hline
5-cycles & $(0,3,4,2,1),$ $(1,3,4,5,2)$  \\\hline
6-cycles & $(0,3,4,5,2,1)$  \\\hline
\end{tabular}
\captionof{table}{Cycles of the $(2,6)$-mother graph.}
\label{2_6_cycles}
\end{center}

We first find all combinations of digit-cycle representations whose lengths add to 5. All cycle combinations which consist solely of 1-cycles will be ignored since they either result in trivial examples or no examples. Also, representations with leading or trailing zeros may be discarded at the discretion of any reader who objects to them. However, we shall include them since, in doing so, finding all 5-digit examples amounts to finding all examples up to 5-digits since any 2, 3, and 4-digit examples will appear in our 5-digit list. For example, the 5-digit permutiple $(0,4,3,1,2)_6 = 2 \cdot (0,2,1,3,4)_6$ reveals a 4-digit example, $(4,3,1,2)_6 = 2 \cdot (2,1,3,4)_6.$

Considering 5-digit $(2,6)$-permutiples, there are several possibilities for digit-cycle combinations whose lengths add to 5: 
\[
\begin{array}{l}
(0,3,4,2,1), (1,3,4,5,2),\\
(1,3,4,2)(0),(1,3,4,2)(5),\\ 
(0,3,1)(1,3), (0,3,1)(2,4),\\ 
(0,3,1)(0)(5), (2,4,5)(1,3),\\ 
(2,4,5)(2,4), (2,4,5)(0)(5),\\
(1,3)(2,4)(0), (1,3)(2,4)(5),\\
(1,3)(1,3)(0),(2,4)(2,4)(0),\\ 
(1,3)(1,3)(5), \mbox{ and } (2,4)(2,4)(5).
\end{array}
\]
Now, in this case, Equation \ref{carries} becomes $\displaystyle \sum_{j=0}^{4}d_j=5\sum_{j=1}^{4} c_j.$ Thus, the sum of the digits of the union of cycles must be divisible by 5. This narrows down our list even further:
$(0,3,4,2,1),$ $(1,3,4,5,2),$ $(1,3,4,2)(0),$ $(1,3,4,2)(5),$ $(0,3,1)(2,4),$ $(2,4,5)(1,3),$ $(1,3)(2,4)(0),$ and $(1,3)(2,4)(5).$  These candidates tell us which collections of digits might yield examples. We shall express these collections as 5-tuples whose entries contain the permuted digits: $(0,1,2,3,4)$  and  $(1,2,3,4,5).$ We may now apply the results of \cite{holt_4} to find permutations of these collections which yield permutiples. The list of all 5-digit $(2,6)$-permutiples in Table \ref{2_6_5_perms} is organized by digit-cycle decomposition. 
\begin{center}
\begin{footnotesize}
 \begin{tabular}{|c|c|}
\hline
 Digit Cycles& Examples\\\hline
$(0,3,4,2,1)$& $(4,2,1,3,0)_6=2\cdot(2,1,0,4,3)_6$ \\
&$(2,4,1,3,0)_6=2\cdot(1,2,0,4,3)_6$\\
&$(4,1,3,0,2)_6=2\cdot(2,0,4,3,1)_6$\\
&$(1,3,0,4,2)_6=2\cdot(0,4,3,2,1)_6$\\
&$(2,1,3,0,4)_6=2\cdot(1,0,4,3,2)_6$\\
&$(1,3,0,2,4)_6=2\cdot(0,4,3,1,2)_6$\\\hline
$(1,3,4,5,2)$&$(5,3,1,4,2)_6=2\cdot(2,4,3,5,1)_6$\\
&$(5,1,3,4,2)_6=2\cdot(2,3,4,5,1)_6$\\
&$(2,5,3,1,4)_6=2\cdot(1,2,4,3,5)_6$\\
&$(2,5,1,3,4)_6=2\cdot(1,2,3,4,5)_6$\\\hline
$(0,3,1)(2,4)$& $(4,3,2,1,0)_6=2\cdot(2,1,4,0,3)_6$\\
&$(3,2,4,1,0)_6=2\cdot(1,4,2,0,3)_6$\\
&$(4,1,2,3,0)_6=2\cdot(2,0,4,1,3)_6$\\
&$(1,2,4,3,0)_6=2\cdot(0,4,2,1,3)_6$\\
&$(4,3,0,1,2)_6=2\cdot(2,1,3,0,4)_6$\\
&$(3,0,4,1,2)_6=2\cdot(1,3,2,0,4)_6$\\
&$(4,1,0,3,2)_6=2\cdot(2,0,3,1,4)_6$\\
&$(1,0,4,3,2)_6=2\cdot(0,3,2,1,4)_6$\\
&$(3,2,1,0,4)_6=2\cdot(1,4,0,3,2)_6$\\
&$(1,2,3,0,4)_6=2\cdot(0,4,1,3,2)_6$\\
&$(3,0,1,2,4)_6=2\cdot(1,3,0,4,2)_6$\\
&$(1,0,3,2,4)_6=2\cdot(0,3,1,4,2)_6$\\\hline
$(2,4,5)(1,3)$&$(5,4,3,1,2)_6=2\cdot(2,5,1,3,4)_6$\\
&$(3,4,5,1,2)_6=2\cdot(1,5,2,3,4)_6$\\
&$(5,1,4,3,2)_6=2\cdot(2,3,5,1,4)_6$\\
&$(3,1,4,5,2)_6=2\cdot(1,3,5,2,4)_6$\\
&$(5,2,3,1,4)_6=2\cdot(2,4,1,3,5)_6$\\
&$(3,2,5,1,4)_6=2\cdot(1,4,2,3,5)_6$\\
&$(5,1,2,3,4)_6=2\cdot(2,3,4,1,5)_6$\\
&$(3,1,2,5,4)_6=2\cdot(1,3,4,2,5)_6$\\\hline
$(1,3)(2,4)(0)$& $(4,3,1,2,0)_6=2\cdot(2,1,3,4,0)_6$\\
&$(3,1,2,4,0)_6=2\cdot(1,3,4,2,0)_6$\\
&$(4,0,3,1,2)_6=2\cdot(2,0,1,3,4)_6$\\
&$(0,4,3,1,2)_6=2\cdot(0,2,1,3,4)_6$\\
&$(3,1,2,0,4)_6=2\cdot(1,3,4,0,2)_6$\\
&$(0,3,1,2,4)_6=2\cdot(0,1,3,4,2)_6$\\\hline
$(1,3)(2,4)(5)$&$(4,3,5,1,2)_6=2\cdot(2,1,5,3,4)_6$\\
&$(4,3,1,5,2)_6=2\cdot(2,1,3,5,4)_6$\\
&$(3,5,1,2,4)_6=2\cdot(1,5,3,4,2)_6$\\
&$(3,1,5,2,4)_6=2\cdot(1,3,5,4,2)_6$\\\hline
\end{tabular}
\end{footnotesize}
\captionof{table}{The collection of all 5-digit $(2,6)$-permutiples.}
\label{2_6_5_perms}
\end{center}
\end{example}

\begin{remark}
From Table \ref{2_6_5_perms}, we may infer that there are no 2- or 3-digit examples. Also, the reader may observe that the collections of permutiples corresponding to the digit-cycle combinations $(1,3,4,5,2),$ $(2,5,4)(1,3),$ and $(1,3)(2,4)(5)$ are the $(2,6)$-permutiple conjugacy classes found in \cite{holt_3} from the single example $(4,3,5,1,2)_6 =
2 \cdot (2,1,5,3,4)_6.$
\end{remark}

\section{Finite-State Machine Methods}

The above results and methods make little mention of the carries in the process of single-digit multiplication. However, in order to broaden our scope, we shall now bring them into the analysis. Hoey \cite{hoey} describes palintiples as a formal language recognized by a finite-state machine, and Sloane's \cite{sloane} Young-graph construction is essentially the state diagram for Hoey's machines. More recent efforts applying similar techniques to problems in number theory include the work of Faber and Grantham \cite{faber} who use finite-state machines (deterministic finite automata) to find integers whose sum is the reverse of their product in an arbitrary base. We now cast the general permutiple problem into a similar light. We note that our construction will be somewhat different from those found in \cite{hoey, sloane} since we are not working with any particular digit permutation. 

Before we begin, we note that for any permutiple carry, $0\leq c_j \leq n-1,$ and any base-$b$ digits, $d_1$ and $d_2,$ we have that $-(b-1)\leq nd_{2}-d_1+c_j \leq n(b-1)+n-1,$ so that $-(b-1)\leq bc_{j+1} \leq nb-1,$ which guarantees that $0\leq c_{j+1} \leq n-1.$ We may now describe a finite-state machine, which, among other numbers, recognizes $(n,b)$-permutiples. Taking non-negative integers less than $n$ as the collection of states, and the edges of $M$ as the input alphabet, the equation
\begin{equation}
 c_2=\left[nd_{2}-d_1+c_1\right]\div b
 \label{transition} 
\end{equation}
defines a state-transition function from state $c_1$ to state $c_2$ with $(d_1,d_{2})$ serving as the input which induces the transition. This transition corresponds to a labeled edge on the state diagram as seen in Figure \ref{stae_diagram}. 
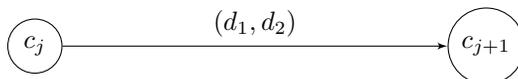
\begin{center}
\begin{tikzpicture}
\tikzset{edge/.style = {->,> = latex'}}
\tikzset{vertex/.style = {shape=circle,draw,minimum size=1.5em}}
[xscale=2, yscale=2, auto=left,every node/.style={circle,fill=blue!20}]
\node[vertex] (n0) at (7,5) {$c_1$};
\node[vertex] (n1) at (13,5) {$c_2$};
\draw[edge] (n0) edge node[above] {$(d_1,d_2)$} (n1);
\end{tikzpicture}
\captionof{figure}{An edge on the state diagram.}
\label{stae_diagram}
\end{center}
Since the first carry, $c_0,$ of any single-digit multiplication is defined to be zero, the initial state must always be zero. Also, for an $\ell$-digit $(n,b)$-permutiple, $c_{\ell}$ must also be zero, otherwise, the result would be an $(\ell+1)$-digit number. That is, zero must be the only accepting state. We shall call the above construction the $(n,b)$-{\it Hoey-Sloane machine}, and we shall call its state diagram the $(n,b)$-{\it Hoey-Sloane graph}, which we will denote as $\Gamma.$ 

Here, we point out that, rather than strings of base-$b$ digits, the $(n,b)$-Hoey-Sloane machine accepts strings of edges from $M$ as inputs. Furthermore, since multiple digit pairs, $(d_1,d_{2}),$ can solve Equation (\ref{transition}) for particular values of $c_1$ and $c_2,$ there are generally multiple inputs which the machine will accept for a transition to occur. Thus, a collection of inputs (a subset of the edges of $M$) is assigned to each edge on $\Gamma$ by the mapping  $(c_1,c_2) \mapsto \{(d_1,d_2)\in E_M|c_2=\left[nd_{2}-d_1+c_1\right]\div b\},$ where $E_M$ is the collection of edges of $M.$ For simpler figures, we will leave off the set braces in figures depicting $\Gamma$ and its subgraphs.   

We shall denote the regular language of input strings accepted by the $(n,b)$-Hoey-Sloane machine as $L.$ We may describe $L$ as finite sequences of edge-label inputs which define walks on $\Gamma$ whose initial and final states are zero. For simplicity, we will call such walks {\it $L$-walks}. Members of $L$ which produce permutiple numbers will be called {\it $(n,b)$-permutiple strings.} When the context is clear, we may omit the ``$(n,b)$-'' prefix for smoother exposition.

\begin{example}\label{two_four_0}
 We may describe the language of input strings, $L,$ recognized by the $(2,4)$-Hoey-Sloane machine by the $(2,4)$-Hoey-Sloane graph, $\Gamma,$ seen in Figure \ref{2_4_hsg}.
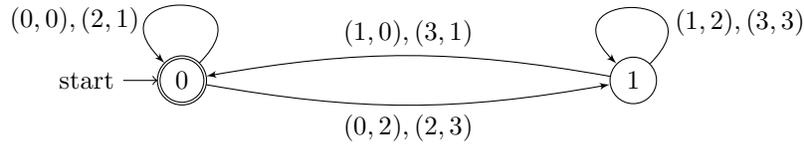
\begin{center}
\begin{tikzpicture}
\tikzset{edge/.style = {->,> = latex'}}
\tikzset{vertex/.style = {shape=circle,draw,minimum size=1.5em}}    
[xscale=2, yscale=2, auto=left,every node/.style={circle,fill=blue!20}]
\node[vertex,accepting,initial] (n0) at (7,5) {$0$};
\node[vertex] (n1) at (13,5) {$1$};
\draw[edge, bend right=10] (n0) edge node[below] {$(0,2),(2,3)$} (n1);
\draw[edge, bend right=10] (n1) edge node[above] {$(1,0),(3,1)$} (n0);
\draw[edge](n0) to[loop] node[pos=0.75,left] {$(0,0),(2,1)$} (n0);
\draw[edge](n1) to[loop] node[pos=0.25,right] {$(1,2),(3,3)$} (n1);
\end{tikzpicture}
\captionof{figure}{The $(2,4)$-Hoey-Sloane graph.}
\label{2_4_hsg}
\end{center}
Clearly, all $(2,4)$-permutiple strings are members of $L.$ For instance, the $(2,4)$-permutiple $(3, 1, 2, 0)_4 = 2 \cdot (1, 2, 3, 0)_4$ corresponds to the $(2,4)$-permutiple string $(0,0)(2,3)(1,2)(3,1),$ which is certainly a member of $L.$ The corresponding $L$-walk is visualized in Figure \ref{2_4_L_walk}. 
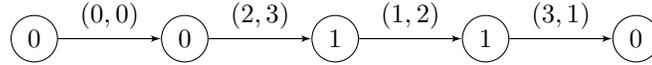
\begin{center}
\begin{tikzpicture}
\tikzset{edge/.style = {->,> = latex'}}
\tikzset{vertex/.style = {shape=circle,draw,minimum size=1.5em}}
[xscale=2, yscale=2, auto=left,every node/.style={circle,fill=blue!20}]
\node[vertex] (n0) at (0,5) {$0$};
\node[vertex] (n1) at (2,5) {$0$};
\node[vertex] (n2) at (4,5) {$1$};
\node[vertex] (n3) at (6,5) {$1$};
\node[vertex] (n4) at (8,5) {$0$};
\draw[edge] (n0) edge node[above] {$(0,0)$} (n1);
\draw[edge] (n1) edge node[above] {$(2,3)$} (n2);
\draw[edge] (n2) edge node[above] {$(1,2)$} (n3);
\draw[edge] (n3) edge node[above] {$(3,1)$} (n4);
\end{tikzpicture}
\captionof{figure}{The $L$-walk of the $(2,4)$-permutiple string $(0,0)(2,3)(1,2)(3,1).$}
\label{2_4_L_walk}
\end{center}
\end{example}

Within this new framework, Theorem \ref{cycle_union} gives us a corollary.

\begin{corollary}\label{cycle_union_2}
Let $s=(d_{0},\hat{d_{0}})(d_{1},\hat{d_{1}})\cdots(d_{k},\hat{d_{k}})$ be a member of $L.$ If $s$ is a permutiple string, then the collection of ordered-pair inputs of $s$ is a union of cycles of $M.$
\end{corollary}

Here, we underscore both that $L$ also consists of input strings which are not permutiple strings, and that the converse of the above is not generally true. As an example, $(2,1)(0,2)(1,2)(1,0)$ is a member of $L$ corresponding to the multiplication $(1,1,0,2)_4=2\cdot(0,2,2,1)_4,$ yet this is not a $(2,4)$-permutiple. Moreover, these inputs form a union of cycles on the $(2,4)$-mother graph, showing that the converse of Corollary \ref{cycle_union_2} does not hold. With a simple modification, however, the above input string can be extended to $(2,1)(0,2)(1,2)(1,0)(2,1),$ which yields a permutiple,  $(2,1,1,0,2)_4=2\cdot(1,0,2,2,1)_4.$  Another non-permutiple example, $(2,2,1,1,0,2)_4=2\cdot(1,1,0,2,2,1)_4,$ bears strong resemblance to a $(2,4)$-permutiple, and is generated from the input string $(2,1)(0,2)(1,2)(1,0)(2,1)(2,1).$

Corollary \ref{cycle_union_2} tells us that the cycles of $M$ play a key role in our search for permutiple strings, but the presence of cycles alone is not enough to guarantee that an input string in $L$ is a permutiple string. In general, we see that for an input string, $s=(d_{0},\hat{d_{0}})(d_{1},\hat{d_{1}})\cdots(d_{k},\hat{d_{k}}),$ in $L$ to be a permutiple string, it must be that the two strings of base-$b$ digits formed by the left and right components of the inputs of $s,$ that is, $d_{k}\cdots d_{1} d_{0}$ and $\hat{d_{k}}\cdots \hat{d_{1}} \hat{d_{0}},$ are permutations of one another. That is, members of $L$ whose left and right components form the same multiset (allowing for repeat digits to be considered distinct) will result in an $(n,b)$-permutiple. Now, collections of inputs for which the left and right components form the same collection of digits are precisely the cycles of $M.$ It follows that any multiset union of cycles of $M$ which can be ordered into an input string, $s,$ which belongs to $L,$ must result in a permutiple string. Since members of $L$ can be visualized as $L$-walks on the $(n,b)$-Hoey-Sloane graph, $\Gamma,$ we see that the edges of $\Gamma$ associated with the inputs of $s$ must define a strongly-connected subgraph of $\Gamma$ containing the zero state. 

With the above observations, we may now describe a method for identifying multiset unions of mother-graph cycles which may be ordered into permutiple strings. We may examine the subgraphs of $\Gamma$ associated with individual cycles of $M.$ If a union of these subgraphs forms a strongly-connected subgraph of $\Gamma$ which contains the zero state, then the multiset union of individual cycles may be ordered into a permutiple string by forming $L$-walks with the aid of $\Gamma.$  

More precisely, we list all the cycles, $\mathscr{C}=\{C_0,C_1,\ldots,C_m\},$ of $M.$ Each element, $C_j,$ of $\mathscr{C}$ defines a subgraph, $\Gamma_j,$ of $\Gamma$ where each edge of $\Gamma_j$ is assigned the edge-label collection by the mapping
$
(c_1,c_2)\mapsto \left\{(d_1,d_2)\in C_j|c_2=\left[nd_{2}-d_1+c_1\right]\div b\right \}.  
$
Any edge, $(c_1,c_2),$ for which this collection is empty will not be included as an edge on $\Gamma_j.$ With the above edges, any state for which both the indegree and outdegree are zero will not be included as a vertex. Each $\Gamma_j$ will be referred to as {\it the image of $C_j,$} or simply as a {\it cycle image}. Letting $I$ be a multiset whose support is a subset, $J,$ of $\{0,1,\ldots,m\},$ then, if the cycle-image union $\Gamma_{J}=\bigcup_{j\in J}\Gamma_j$ (edge labels included) is a strongly-connected subgraph of $\Gamma$ containing the zero state, then the associated multiset union, $C_{I}=\bigcup_{j\in I}C_j,$ may be ordered into permutiple strings by forming $L$-walks on $\Gamma.$ This is easily accomplished by using $\Gamma_{J}.$

When forming permutiple strings, we emphasize that every element of $C_I$ must be used, repeats and all, otherwise, the multisets of left and right components will not be equal. Also, cycles may be chosen more than once, resulting in a multiset consisting of repeated values, each of which occurs with equal multiplicity.  

With the above, we have everything we need to easily generate examples of $(2,4)$-permutiples.

\begin{example}\label{two_four}
The $(2,4)$-mother graph, $M,$ is given in Figure \ref{2_4_mg}.
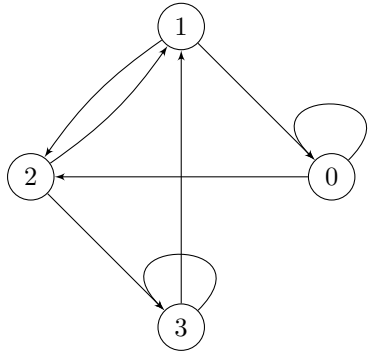
\begin{center}
\begin{tikzpicture} \tikzset{edge/.style = {->,> = latex'}} \tikzset{vertex/.style = {shape=circle,draw,minimum size=1.5em}} [xscale=2, yscale=2, auto=left,every node/.style={circle,fill=blue!20}] 
\node[vertex] (n0) at (12,5) {$0$}; 
\node[vertex] (n1) at (10.000002653589794,6.99999999999824) {$1$}; 
\node[vertex] (n2) at (8.000000000007041,5.0000053071795865) {$2$}; \node[vertex] (n3) at (9.99999203923062,3.0000000000158433) {$3$}; 
\draw[edge](n0) to[loop] (n0); 
\draw[edge](n0) to (n2); 
\draw[edge](n1) to (n0); 
\draw[edge, bend right=10](n1) to (n2); 
\draw[edge, bend right=10](n2) to (n1); 
\draw[edge](n2) to (n3); 
\draw[edge](n3) to (n1); 
\draw[edge](n3) to[loop] (n3); 
\end{tikzpicture}
\captionof{figure}{The $(2,4)$-mother graph.}
\label{2_4_mg}
\end{center}
The cycles, $\mathscr{C},$ written as multisets of ordered-pair inputs, are listed below:
\[
 \begin{array}{l}
C_0=(0) =\{(0,0)\},\\
C_1=(3) =\{(3,3)\},\\
C_2=(1,2) =\{(1,2),(2,1)\},\\
C_3=(0,2,1) =\{(0,2),(2,1)(1,0)\},\\
C_4=(1,2,3) =\{(1,2),(2,3),(3,1)\},\\
C_5=(0,2,3,1) =\{(0,2),(2,3),(3,1),(1,0)\}.\\
\end{array}
\] 
We now map each cycle to its corresponding image as seen in Table \ref{2_4_ci}. 
\begin{center}
\begin{scriptsize}
\begin{tabular}{|c|c|l|c|}
\hline
 & {\bf Mother-Graph Cycle} & {\bf Cycle Image} &\\\hline

$C_0$ & 
\begin{tikzpicture} \tikzset{edge/.style = {->,> = latex'}} \tikzset{vertex/.style = {shape=circle,draw,minimum size=1.5em}} [xscale=2, yscale=2, auto=left,every node/.style={circle,fill=blue!20}] 
\node[vertex] (n0) at (4,5) {$0$}; 
\draw[edge](n0) to[loop] (n0); 
\end{tikzpicture} 
&
\begin{tikzpicture} \tikzset{edge/.style = {->,> = latex'}} \tikzset{vertex/.style = {shape=circle,draw,minimum size=1.5em}} [xscale=2, yscale=2, auto=left,every node/.style={circle,fill=blue!20}] 
\node[vertex,accepting,initial] (n0) at (11,5.5) {$0$};
\node[vertex,white] (n1) at (14,5.5) {$1$};
\draw[edge](n0) to[loop] node[pos=0.75,left] {$(0,0)$} (n0);
\end{tikzpicture} 
&
$\Gamma_0$
\\\hline

$C_1$ &
\begin{tikzpicture} \tikzset{edge/.style = {->,> = latex'}} \tikzset{vertex/.style = {shape=circle,draw,minimum size=1.5em}} [xscale=2, yscale=2, auto=left,every node/.style={circle,fill=blue!20}] 
\node[vertex] (n3) at (5.99999601961531,4.000000000007922) {$3$};
\draw[edge](n3) to[loop] (n3);
\end{tikzpicture} 
&
\color{white}
\begin{tikzpicture} \tikzset{edge/.style = {->,> = latex'}} \tikzset{vertex/.style = {shape=circle,draw,minimum size=1.5em}} [xscale=2, yscale=2, auto=left,every node/.style={circle,fill=blue!20}] 
\node[vertex,accepting,initial] (n0) at (11,5.5) {$0$};
\node[vertex,black] (n1) at (14,5.5) {$1$};
\draw[edge,black](n1) to[loop] node[pos=0.25,right] {$(3,3)$} (n1);
\end{tikzpicture} 
&
$\Gamma_1$
\\\hline

$C_2$ &
\begin{tikzpicture} \tikzset{edge/.style = {->,> = latex'}} \tikzset{vertex/.style = {shape=circle,draw,minimum size=1.5em}} [xscale=2, yscale=2, auto=left,every node/.style={circle,fill=blue!20}] 
\node[vertex] (n1) at (6.000001326794896,5.99999999999912) {$1$}; 
\node[vertex] (n2) at (5.00000000000352,5.000002653589793) {$2$}; 
\draw[edge, bend right=10](n1) to (n2); 
\draw[edge,bend right=10](n2) to (n1); 
\end{tikzpicture} 

&

\begin{tikzpicture} \tikzset{edge/.style = {->,> = latex'}} \tikzset{vertex/.style = {shape=circle,draw,minimum size=1.5em}} [xscale=2, yscale=2, auto=left,every node/.style={circle,fill=blue!20}] 
\node[vertex,accepting,initial] (n0) at (11,5.5) {$0$};
\node[vertex] (n1) at (14,5.5) {$1$};
\draw[edge](n0) to[loop] node[pos=0.75,left] {$(2,1)$} (n0);
\draw[edge](n1) to[loop] node[pos=0.25,right] {$(1,2)$} (n1);
\end{tikzpicture} 
&
$\Gamma_2$
\\\hline

$C_3$ &
\begin{tikzpicture} \tikzset{edge/.style = {->,> = latex'}} \tikzset{vertex/.style = {shape=circle,draw,minimum size=1.5em}} [xscale=2, yscale=2, auto=left,every node/.style={circle,fill=blue!20}] 
\node[vertex] (n0) at (7,5) {$0$}; 
\node[vertex] (n1) at (6.000001326794896,5.99999999999912) {$1$}; 
\node[vertex] (n2) at (5.00000000000352,5.000002653589793) {$2$};  
\draw[edge](n0) to (n2); 
\draw[edge](n1) to (n0); 
\draw[edge](n2) to (n1); 
\end{tikzpicture} 

&

\begin{tikzpicture} \tikzset{edge/.style = {->,> = latex'}} \tikzset{vertex/.style = {shape=circle,draw,minimum size=1.5em}} [xscale=2, yscale=2, auto=left,every node/.style={circle,fill=blue!20}] 
\node[vertex,accepting,initial] (n0) at (11,5.5) {$0$};
\node[vertex] (n1) at (14,5.5) {$1$};
\draw[edge, bend right=10] (n0) edge node[below] {$(0,2)$} (n1);
\draw[edge, bend right=10] (n1) edge node[above] {$(1,0)$} (n0);
\draw[edge](n0) to[loop] node[pos=0.75,left] {$(2,1)$} (n0);
\end{tikzpicture}
&
$\Gamma_3$
\\\hline

$C_4$ &
\begin{tikzpicture} \tikzset{edge/.style = {->,> = latex'}} \tikzset{vertex/.style = {shape=circle,draw,minimum size=1.5em}} [xscale=2, yscale=2, auto=left,every node/.style={circle,fill=blue!20}] 
\node[vertex] (n1) at (6.000001326794896,5.99999999999912) {$1$}; 
\node[vertex] (n2) at (5.00000000000352,5.000002653589793) {$2$}; 
\node[vertex] (n3) at (5.99999203923062,4.0000000000158433) {$3$}; 
\draw[edge](n1) to (n2); 
\draw[edge](n2) to (n3); 
\draw[edge](n3) to (n1); 
\end{tikzpicture} 

&

\begin{tikzpicture} \tikzset{edge/.style = {->,> = latex'}} \tikzset{vertex/.style = {shape=circle,draw,minimum size=1.5em}} [xscale=2, yscale=2, auto=left,every node/.style={circle,fill=blue!20}] 

\node[vertex,accepting,initial] (n0) at (11,5.5) {$0$};
\node[vertex] (n1) at (14,5.5) {$1$};
\draw[edge, bend right=10] (n0) edge node[below] {$(2,3)$} (n1);
\draw[edge, bend right=10] (n1) edge node[above] {$(3,1)$} (n0);
\draw[edge](n1) to[loop] node[pos=0.25,right] {$(1,2)$} (n1);
\end{tikzpicture} 
&
$\Gamma_4$
\\\hline

$C_5$ &
\begin{tikzpicture} \tikzset{edge/.style = {->,> = latex'}} \tikzset{vertex/.style = {shape=circle,draw,minimum size=1.5em}} [xscale=2, yscale=2, auto=left,every node/.style={circle,fill=blue!20}] 
\node[vertex] (n0) at (7,5) {$0$}; 
\node[vertex] (n1) at (6.000001326794896,5.99999999999912) {$1$}; 
\node[vertex] (n2) at (5.00000000000352,5.000002653589793) {$2$}; 
\node[vertex] (n3) at (5.99999203923062,4.0000000000158433) {$3$}; 
\draw[edge](n0) to (n2); 
\draw[edge](n1) to (n0); 
\draw[edge](n2) to (n3); 
\draw[edge](n3) to (n1); 
\end{tikzpicture} 

&

\begin{tikzpicture} \tikzset{edge/.style = {->,> = latex'}} \tikzset{vertex/.style = {shape=circle,draw,minimum size=1.5em}} [xscale=2, yscale=2, auto=left,every node/.style={circle,fill=blue!20}] 
\node[vertex,accepting,initial] (n0) at (11,5.5) {$0$};
\node[vertex] (n1) at (14,5.5) {$1$};
\draw[edge, bend right=10] (n0) edge node[below] {$(0,2),(2,3)$} (n1);
\draw[edge, bend right=10] (n1) edge node[above] {$(1,0),(3,1)$} (n0);
\end{tikzpicture} 
&
$\Gamma_5$
\\\hline
\end{tabular}
\captionof{table}{Cycles of the $(2,4)$-mother graph and their corresponding cycle images.}
\label{2_4_ci}
\end{scriptsize}
\end{center}
Any union of the cycle images, $\Gamma_j,$ in Table \ref{2_4_ci} which results in a strongly-connected graph containing the zero state corresponds to a multiset union of cycles, $C_j,$  from which we may generate permutiple strings by ordering the inputs into $L$-walks on $\Gamma,$ seen in Figure \ref{2_4_hsg}. In this way, the above gives a complete description of how $(2,4)$-permutiple strings are constructed. 

For example, it is plain to see that the images of each individual multiset, $C_0,$ $C_3,$ $C_4,$ and $C_5,$ are strongly connected and contain the zero state. That is, these cycles individually contain inputs which enable $L$-walks on $\Gamma,$ allowing us to form permutiple strings. Examples are included in Table \ref{2_4_indiv}.
\begin{center}
 \begin{tabular}{|c|l|l|}
 \hline
 Cycle & Permutiple String & Example\\\hline
$C_0$ & $(0,0)$ &$(0)_4=2\cdot(0)_4$\\\hline
$C_3$ & $(2,1)(0,2)(1,0)$ & $(1,0,2)_4=2\cdot(0,2,1)_4$\\
& $(0,2)(1,0)(2,1)$ &$(2,1,0)_4=2\cdot(1,0,2)_4$\\\hline
$C_4$ & $(2,3)(1,2)(3,1)$ & $(3,1,2)_4=2\cdot(1,2,3)_4$\\\hline
$C_5$ & $(0,2)(1,0)(2,3)(3,1)$& $(3,2,1,0)_4=2\cdot(1,3,0,2)_4$\\
 & $(0,2)(3,1)(2,3)(1,0)$& $(1,2,3,0)_4=2\cdot(0,3,1,2)_4$\\
 & $(2,3)(1,0)(0,2)(3,1)$& $(3,0,1,2)_4=2\cdot(1,2,0,3)_4$\\
 & $(2,3)(3,1)(0,2)(1,0)$& $(1,0,3,2)_4=2\cdot(0,2,1,3)_4$\\\hline
\end{tabular} 
\captionof{table}{Examples of $(2,4)$-permutiples corresponding to single mother-graph cycles, $C_0,$ $C_3,$ $C_4,$ and $C_5.$}
\label{2_4_indiv}
\end{center}
It is also clear that unions involving only $C_1$ or $C_2$ are insufficient for this same purpose; $\Gamma_1$ is strongly connected, but does not contain the zero state, and $\Gamma_2$ is not strongly connected. We also see that any union involving the multisets $C_3,$ $C_4,$ and $C_5$ can be ordered to form $L$-walks on $\Gamma$ since the associated cycle-image union is strongly connected and contains the zero state. For example, the multiset union $C_2 \cup C_3=\{(1,2),(2,1),(0,2),(2,1),(1,0)\}$ gives a collection of inputs with which we may form permutiple strings. Examples are included in Table \ref{2_4_union_1}
\begin{center}
 \begin{tabular}{|c|l|l|}
 \hline
 Cycle & Permutiple String & Example\\\hline
$C_2 \cup C_3$ & $(2,1)(0,2)(1,2)(1,0)(2,1)$ &$(2,1,1,0,2)_4=2\cdot(1,0,2,2,1)_4$\\
& $(2,1)(2,1)(0,2)(1,2)(1,0)$ &$(1,1,0,2,2)_4=2\cdot(0,2,2,1,1)_4$\\
& $(0,2)(1,2)(1,0)(2,1)(2,1)$ &$(2,2,1,1,0)_4=2\cdot(1,1,0,2,2)_4$\\\hline
\end{tabular}
\captionof{table}{Examples of $(2,4)$-permutiples corresponding to the multiset union of mother-graph cycles $C_2 \cup C_3.$}
\label{2_4_union_1}
\end{center}
As mentioned above, cycles may be used more than once. For example, the multiset union $C_3 \cup C_3=\{(0,2),(2,1),(1,0),(0,2),(2,1),(1,0)\}$ gives us several six-digit examples, which are displayed in Table \ref{2_4_union_2}.
\begin{center}
 \begin{tabular}{|c|l|l|}
 \hline
Cycle & Permutiple String & Example\\\hline
$C_3 \cup C_3$ & $(2,1)(0,2)(1,0)(2,1)(0,2)(1,0)$ &$(1,0,2,1,0,2)_4=2\cdot(0,2,1,0,2,1)_4$\\
& $(0,2)(1,0)(2,1)(0,2)(1,0)(2,1)$ & $(2,1,0,2,1,0)_4=2\cdot(1,0,2,1,0,2)_4$\\
& $ (2,1)(0,2)(1,0)(0,2)(1,0)(2,1)$ & $(2,1,0,1,0,2)_4=2\cdot(1,0,2,0,2,1)_4$\\
& $ (2,1)(2,1)(0,2)(1,0)(0,2)(1,0)$ & $(1,0,1,0,2,2)_4=2\cdot(0,2,0,2,1,1)_4$\\
& $(0,2)(1,0)(0,2)(1,0)(2,1)(2,1)$ & $(2,2,1,0,1,0)_4=2\cdot(1,1,0,2,0,2)_4$\\
& $(0,2)(1,0)(2,1)(2,1)(0,2)(1,0)$ & $(1,0,2,2,1,0)_4=2\cdot(0,2,1,1,0,2)_4$\\
\hline
\end{tabular}
\captionof{table}{Examples of $(2,4)$-permutiples corresponding to the multiset union of mother-graph cycles $C_3 \cup C_3.$}
\label{2_4_union_2}
\end{center}
In the above way, all nontrivial $(2,4)$-permutiple strings can be formed from multiset unions of cycles of $M$ which contain at least one copy of $C_3,$ $C_4,$ or $C_5.$ 
\end{example}

We now perform a similar analysis to find examples of $(3,4)$-permutiples.

\begin{example}\label{three_four}
 We begin by examining the $(3,4)$-mother graph, $M,$ displayed in Figure \ref{3_4_mg}.
 \begin{center}
\begin{tikzpicture} 
\tikzset{edge/.style = {->,> = latex'}} \tikzset{vertex/.style = {shape=circle,draw,minimum size=1.5em}} [xscale=2, yscale=2, auto=left,every node/.style={circle,fill=blue!20}] \node[vertex] (n0) at (12,5) {$0$}; 
\node[vertex] (n1) at (10.000002653589794,6.99999999999824) {$1$}; \node[vertex] (n2) at (8.000000000007041,5.0000053071795865) {$2$}; \node[vertex] (n3) at (9.99999203923062,3.0000000000158433) {$3$}; 
\draw[edge](n0) to[loop] (n0); 
\draw[edge, bend right=10](n0) to (n1); 
\draw[edge, bend right=10](n0) to (n2); 
\draw[edge, bend right=10](n1) to (n0); 
\draw[edge](n1) to[loop] (n1); 
\draw[edge, bend right=10](n1) to (n3); 
\draw[edge, bend right=10](n2) to (n0); 
\draw[edge](n2) to[loop] (n2); 
\draw[edge, bend right=10](n2) to (n3); 
\draw[edge, bend right=10](n3) to (n1); 
\draw[edge, bend right=10](n3) to (n2); 
\draw[edge](n3) to[loop] (n3); 
\end{tikzpicture}
\captionof{figure}{The $(3,4)$-mother graph.}
\label{3_4_mg}.
\end{center}
The collection of cycles, $\mathscr{C},$ is listed below:
\[
 \begin{array}{l}
C_0=(0) =\{(0,0)\},\\
C_1=(1) =\{(1,1)\},\\
C_2=(2) =\{(2,2)\},\\
C_3=(3) =\{(3,3)\},\\
C_4=(0,1) =\{(0,1),(1,0)\},\\
C_5=(0,2) =\{(0,2),(2,0)\},\\
C_6=(1,3) =\{(1,3),(3,1)\},\\
C_7=(2,3) =\{(2,3),(3,2)\},\\
C_8=(0,1,3,2) =\{(0,1),(1,3),(3,2),(2,0)\},\\
C_9=(0,2,3,1) =\{(0,2),(2,3),(3,1),(1,0)\}.\\
\end{array}
\]
We now look for multiset unions of the above cycles from which we may construct members of $L.$ The $(3,4)$-Hoey-Sloane graph, $\Gamma,$ is seen in Figure \ref{3_4_hsg}.
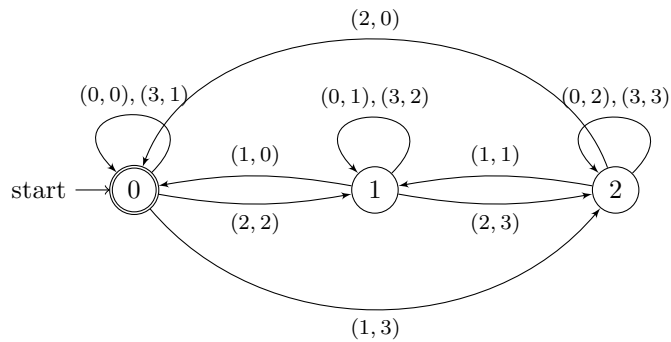
\begin{center}
\begin{tikzpicture}
\tikzset{edge/.style = {->,> = latex'}}
\tikzset{vertex/.style = {shape=circle,draw,minimum size=1.5em}}
[xscale=2, yscale=2, auto=left,every node/.style={circle,fill=blue!20}]
\node[vertex,initial,accepting] (n0) at (0,5) {$0$};
\node[vertex] (n1) at (3.2,5) {$1$};
\node[vertex] (n2) at (6.4,5) {$2$};
\draw[edge](n0) to[loop] node[above] {\footnotesize{$(0,0),(3,1)$}} (n0);
\draw[edge](n1) to[loop] node[above] {\footnotesize{$(0,1),(3,2)$}} (n1);
\draw[edge](n2) to[loop] node[above] {\footnotesize{$(0,2),(3,3)$}} (n2);
\draw[edge, bend right=10](n0) edge node[below] {\footnotesize{$(2,2)$}} (n1);
\draw[edge, bend right=10](n1) edge node[above] {\footnotesize{$(1,0)$}} (n0);
\draw[edge, bend right=10](n2) edge node[above] {\footnotesize{$(1,1)$}} (n1);
\draw[edge, bend right=10](n1) edge node[below] {\footnotesize{$(2,3)$}} (n2);
\draw[edge, bend right=50](n0) edge node[below] {\footnotesize{$(1,3)$}} (n2);
\draw[edge, bend right=70](n2) edge node[above] {\footnotesize{$(2,0)$}} (n0);
\end{tikzpicture}
\captionof{figure}{The $(3,4)$-Hoey-Sloane graph.}
\label{3_4_hsg}
\end{center}
To accomplish the objective stated above, we compare individual mother-graph cycles to their images as seen in Table \ref{3_4_ci}.
\begin{center}
\begin{scriptsize}
\begin{tabular}{|c|c|l|c|}
\hline
& {\bf Mother-Graph Cycle} & {\bf Cycle Image}&\\\hline

$C_0$ &
\begin{tikzpicture} \tikzset{edge/.style = {->,> = latex'}} \tikzset{vertex/.style = {shape=circle,draw,minimum size=1.5em}} [xscale=2, yscale=2, auto=left,every node/.style={circle,fill=blue!20}] 
\node[vertex] (n0) at (7,5) {$0$}; 
\draw[edge](n0) to[loop] (n0); 
\end{tikzpicture} 

&

\begin{tikzpicture} \tikzset{edge/.style = {->,> = latex'}} \tikzset{vertex/.style = {shape=circle,draw,minimum size=1.5em}} [xscale=2, yscale=2, auto=left,every node/.style={circle,fill=blue!20}] 
\node[vertex,initial,accepting] (n0) at (0,5) {$0$};
\node[vertex,white] (n1) at (2,5) {$1$};
\node[vertex,white] (n2) at (4,5) {$2$};
\draw[edge](n0) to[loop] node[right,pos=0.25] {\footnotesize $(0,0)$} (n0);
\end{tikzpicture}
&
$\Gamma_0$
\\\hline

$C_1$ &
\begin{tikzpicture} \tikzset{edge/.style = {->,> = latex'}} \tikzset{vertex/.style = {shape=circle,draw,minimum size=1.5em}} [xscale=2, yscale=2, auto=left,every node/.style={circle,fill=blue!20}] 
\node[vertex] (n1) at (6.000001326794896,5.99999999999912) {$1$};  
\draw[edge](n1) to[loop] (n1); 
\end{tikzpicture} 

&

\begin{tikzpicture} \tikzset{edge/.style = {->,> = latex'}} \tikzset{vertex/.style = {shape=circle,draw,minimum size=1.5em}} [xscale=2, yscale=2, auto=left,every node/.style={circle,fill=blue!20}] 
\color{white}\node[vertex,initial,accepting] (n0) at (0,5) {$0$};
\node[vertex,black] (n1) at (2,5) {$1$};
\node[vertex,black] (n2) at (4,5) {$2$};
 \draw[edge,black,bend right=10](n2) edge node[above] {\footnotesize $(1,1)$} (n1);
\end{tikzpicture} 
&
$\Gamma_1$
\\\hline

$C_2$ &
\begin{tikzpicture} \tikzset{edge/.style = {->,> = latex'}} \tikzset{vertex/.style = {shape=circle,draw,minimum size=1.5em}} [xscale=2, yscale=2, auto=left,every node/.style={circle,fill=blue!20}]  
\node[vertex] (n2) at (5.00000000000352,5.000002653589793) {$2$}; 
 \draw[edge](n2) to[loop] (n2); 
\end{tikzpicture} 

&

\begin{tikzpicture} \tikzset{edge/.style = {->,> = latex'}} \tikzset{vertex/.style = {shape=circle,draw,minimum size=1.5em}} [xscale=2, yscale=2, auto=left,every node/.style={circle,fill=blue!20}] 
\node[vertex,initial,accepting] (n0) at (0,5) {$0$};
\node[vertex] (n1) at (2,5) {$1$};
\node[vertex,white] (n2) at (4,5) {$2$};
 \draw[edge, bend right=10](n0) edge node[below] {\footnotesize $(2,2)$} (n1);
\end{tikzpicture}
&
$\Gamma_2$
\\\hline

$C_3$ &
\begin{tikzpicture} \tikzset{edge/.style = {->,> = latex'}} \tikzset{vertex/.style = {shape=circle,draw,minimum size=1.5em}} [xscale=2, yscale=2, auto=left,every node/.style={circle,fill=blue!20}] 
 \node[vertex] (n3) at (5.99999203923062,4.0000000000158433) {$3$};  
\draw[edge](n3) to[loop] (n3);
\end{tikzpicture} 

&
\color{white}
\begin{tikzpicture} \tikzset{edge/.style = {->,> = latex'}} \tikzset{vertex/.style = {shape=circle,draw,minimum size=1.5em}} [xscale=2, yscale=2, auto=left,every node/.style={circle,fill=blue!20}] 
\node[vertex,initial,accepting] (n0) at (0,5) {$0$};
\node[vertex] (n1) at (2,5) {$1$};
\node[vertex,black] (n2) at (4,5) {$2$};
 \draw[edge,black](n2) to[loop] node[left,pos=0.75] {\footnotesize $(3,3)$} (n2);
\end{tikzpicture}
&
$\Gamma_3$
\\\hline

$C_4$ &
\begin{tikzpicture} \tikzset{edge/.style = {->,> = latex'}} \tikzset{vertex/.style = {shape=circle,draw,minimum size=1.5em}} [xscale=2, yscale=2, auto=left,every node/.style={circle,fill=blue!20}] 
\node[vertex] (n0) at (7,5) {$0$}; 
\node[vertex] (n1) at (6.000001326794896,5.99999999999912) {$1$};  
\draw[edge, bend right=10](n1) to (n0); 
\draw[edge, bend right=10](n0) to (n1); 
\end{tikzpicture} 

&

\begin{tikzpicture} \tikzset{edge/.style = {->,> = latex'}} \tikzset{vertex/.style = {shape=circle,draw,minimum size=1.5em}} [xscale=2, yscale=2, auto=left,every node/.style={circle,fill=blue!20}] 
\node[vertex,initial,accepting] (n0) at (0,5) {$0$};
\node[vertex] (n1) at (2,5) {$1$};
\node[vertex,white] (n2) at (4,5) {$2$};
\draw[edge](n1) to[loop] node[above] {\footnotesize $(0,1)$} (n1);
\draw[edge, bend right=10](n1) edge node[above] {\footnotesize $(1,0)$} (n0);
\end{tikzpicture}
&
$\Gamma_4$
\\\hline

$C_5$ &
\begin{tikzpicture} \tikzset{edge/.style = {->,> = latex'}} \tikzset{vertex/.style = {shape=circle,draw,minimum size=1.5em}} [xscale=2, yscale=2, auto=left,every node/.style={circle,fill=blue!20}] 
\node[vertex] (n0) at (7,5) {$0$}; 
\node[vertex] (n2) at (5.00000000000352,5.000002653589793) {$2$};  
\draw[edge, bend right=10](n0) to (n2); 
\draw[edge, bend right=10](n2) to (n0); 
\end{tikzpicture} 

&

\begin{tikzpicture} \tikzset{edge/.style = {->,> = latex'}} \tikzset{vertex/.style = {shape=circle,draw,minimum size=1.5em}} [xscale=2, yscale=2, auto=left,every node/.style={circle,fill=blue!20}] 
\node[vertex,initial,accepting] (n0) at (0,5) {$0$};
\node[vertex,white] (n1) at (2,5) {$1$};
\node[vertex] (n2) at (4,5) {$2$};
\draw[edge](n2) to[loop] node[right,pos=0.25] {\footnotesize $(0,2)$} (n2);
\draw[edge, bend right=20](n2) edge node[above] {\footnotesize $(2,0)$} (n0);
\end{tikzpicture}
&
$\Gamma_5$
\\\hline

$C_6$ &
\begin{tikzpicture} \tikzset{edge/.style = {->,> = latex'}} \tikzset{vertex/.style = {shape=circle,draw,minimum size=1.5em}} [xscale=2, yscale=2, auto=left,every node/.style={circle,fill=blue!20}] 
\node[vertex] (n1) at (6.000001326794896,5.99999999999912) {$1$}; 
\node[vertex] (n3) at (5.99999203923062,4.0000000000158433) {$3$};  
\draw[edge, bend right=10](n3) to (n1); 
\draw[edge, bend right=10](n1) to (n3); 
\end{tikzpicture} 

&

\begin{tikzpicture} \tikzset{edge/.style = {->,> = latex'}} \tikzset{vertex/.style = {shape=circle,draw,minimum size=1.5em}} [xscale=2, yscale=2, auto=left,every node/.style={circle,fill=blue!20}] 
\node[vertex,initial,accepting] (n0) at (0,5) {$0$};
\node[vertex,white] (n1) at (2,5) {$1$};
\node[vertex] (n2) at (4,5) {$2$};
\draw[edge](n0) to[loop] node[above] {\footnotesize $(3,1)$} (n0);
\draw[edge, bend right=20](n0) edge node[below] {\footnotesize $(1,3)$} (n2);
\end{tikzpicture} 
&
$\Gamma_6$
\\\hline

$C_7$ &
\begin{tikzpicture} \tikzset{edge/.style = {->,> = latex'}} \tikzset{vertex/.style = {shape=circle,draw,minimum size=1.5em}} [xscale=2, yscale=2, auto=left,every node/.style={circle,fill=blue!20}] 
\node[vertex] (n2) at (5.00000000000352,5.000002653589793) {$2$}; 
\node[vertex] (n3) at (5.99999203923062,4.0000000000158433) {$3$}; 
\draw[edge, bend right=10](n2) to (n3); 
\draw[edge, bend right=10](n3) to (n2); 
\end{tikzpicture} 

&
\color{white}
\begin{tikzpicture} \tikzset{edge/.style = {->,> = latex'}} \tikzset{vertex/.style = {shape=circle,draw,minimum size=1.5em}} [xscale=2, yscale=2, auto=left,every node/.style={circle,fill=blue!20}] 
\node[vertex,initial,accepting] (n0) at (0,5) {$0$};
\node[vertex,black] (n1) at (2,5) {$1$};
\node[vertex,black] (n2) at (4,5) {$2$};
\draw[edge,black](n1) to[loop] node[left,pos=0.75] {\footnotesize $(3,2)$} (n1);
\draw[edge, black, bend right=10](n1) edge node[below] {\footnotesize $(2,3)$} (n2);
\end{tikzpicture}
&
$\Gamma_7$
\\\hline

$C_8$ &
\begin{tikzpicture} \tikzset{edge/.style = {->,> = latex'}} \tikzset{vertex/.style = {shape=circle,draw,minimum size=1.5em}} [xscale=2, yscale=2, auto=left,every node/.style={circle,fill=blue!20}] 
\node[vertex] (n0) at (7,5) {$0$}; 
\node[vertex] (n1) at (6.000001326794896,5.99999999999912) {$1$}; 
\node[vertex] (n2) at (5.00000000000352,5.000002653589793) {$2$}; 
\node[vertex] (n3) at (5.99999203923062,4.0000000000158433) {$3$};  
\draw[edge](n2) to (n0); 
\draw[edge](n0) to (n1); 
\draw[edge](n3) to (n2); 
\draw[edge](n1) to (n3); 
\end{tikzpicture} 

&

\begin{tikzpicture} \tikzset{edge/.style = {->,> = latex'}} \tikzset{vertex/.style = {shape=circle,draw,minimum size=1.5em}} [xscale=2, yscale=2, auto=left,every node/.style={circle,fill=blue!20}] 
\node[vertex,initial,accepting] (n0) at (0,5) {$0$};
\node[vertex] (n1) at (2,5) {$1$};
\node[vertex] (n2) at (4,5) {$2$};
\draw[edge](n1) to[loop] node[above] {\footnotesize $(0,1),(3,2)$} (n1);
\draw[edge, bend right=20](n0) edge node[below] {\footnotesize $(1,3)$} (n2);
\draw[edge, bend right=80](n2) edge node[above] {\footnotesize $(2,0)$} (n0);
\end{tikzpicture} 
&
$\Gamma_8$
\\\hline

$C_9$ &
\begin{tikzpicture} \tikzset{edge/.style = {->,> = latex'}} \tikzset{vertex/.style = {shape=circle,draw,minimum size=1.5em}} [xscale=2, yscale=2, auto=left,every node/.style={circle,fill=blue!20}] 
\node[vertex] (n0) at (7,5) {$0$}; 
\node[vertex] (n1) at (6.000001326794896,5.99999999999912) {$1$}; 
\node[vertex] (n2) at (5.00000000000352,5.000002653589793) {$2$}; 
\node[vertex] (n3) at (5.99999203923062,4.0000000000158433) {$3$};  
\draw[edge](n0) to (n2); 
\draw[edge](n1) to (n0); 
\draw[edge](n2) to (n3); 
\draw[edge](n3) to (n1); 
\end{tikzpicture} 

&

\begin{tikzpicture} \tikzset{edge/.style = {->,> = latex'}} \tikzset{vertex/.style = {shape=circle,draw,minimum size=1.5em}} [xscale=2, yscale=2, auto=left,every node/.style={circle,fill=blue!20}] 
\node[vertex,initial,accepting] (n0) at (0,5) {$0$};
\node[vertex] (n1) at (2,5) {$1$};
\node[vertex] (n2) at (4,5) {$2$};
\draw[edge](n0) to[loop] node[above] {\footnotesize $(3,1)$} (n0);
\draw[edge](n2) to[loop] node[above] {\footnotesize $(0,2)$} (n2);
\draw[edge, bend right=10](n1) edge node[above] {\footnotesize $(1,0)$} (n0);
\draw[edge, bend right=10](n1) edge node[below] {\footnotesize $(2,3)$} (n2);
\end{tikzpicture} 
&
$\Gamma_9$
\\\hline

\end{tabular}
\captionof{table}{Cycles of the $(3,4)$-mother graph and their corresponding cycle images.}
\label{3_4_ci}
\end{scriptsize}
\end{center}
We may now consider multiset unions of mother-graph cycles whose corresponding cycle-image union results in a strongly-connected graph containing the zero state. In this way, Table \ref{3_4_ci} describes how to form any $(3,4)$-permutiple string. For instance, the cycle union $C_5 \cup C_6=\{(0,2),(2,0),(1,3),(3,1)\}$ corresponds to the strongly-connected cycle-image union $\Gamma_5\cup \Gamma_6,$ which is shown in Figure \ref{3_4_union}.
\begin{center}
\begin{tikzpicture} \tikzset{edge/.style = {->,> = latex'}} \tikzset{vertex/.style = {shape=circle,draw,minimum size=1.5em}} [xscale=2, yscale=2, auto=left,every node/.style={circle,fill=blue!20}] 
\node[vertex,initial,accepting] (n0) at (0,5) {$0$};
\node[vertex,white] (n1) at (3.5,5) {$1$};
\node[vertex] (n2) at (7,5) {$2$};
\draw[edge](n0) to[loop] node[above] {$(3,1)$} (n0);
\draw[edge](n2) to[loop] node[above] {$(0,2)$} (n2);
\draw[edge, bend right=10](n0) edge node[below] { $(1,3)$} (n2);
\draw[edge, bend right=10](n2) edge node[above] {$(2,0)$} (n0);
\end{tikzpicture}
\captionof{figure}{The cycle-image union $\Gamma_5\cup \Gamma_6$ corresponding to the mother-graph cycle union $C_5 \cup C_6.$}
\label{3_4_union}
\end{center}
From Figure \ref{3_4_union}, we may easily generate $L$-walks on $\Gamma,$ two of which are displayed in Figure \ref{3_4_L_walks}.
\begin{center}
\begin{tikzpicture}
\tikzset{edge/.style = {->,> = latex'}}
\tikzset{vertex/.style = {shape=circle,draw,minimum size=1.5em}}
[xscale=2, yscale=2, auto=left,every node/.style={circle,fill=blue!20}]
\node[vertex] (n0) at (0,5) {$0$};
\node[vertex] (n1) at (2,5) {$2$};
\node[vertex] (n2) at (4,5) {$2$};
\node[vertex] (n3) at (6,5) {$0$};
\node[vertex] (n4) at (8,5) {$0$};
\draw[edge] (n0) edge node[above] {$(1,3)$} (n1);
\draw[edge] (n1) edge node[above] {$(0,2)$} (n2);
\draw[edge] (n2) edge node[above] {$(2,0)$} (n3);
\draw[edge] (n3) edge node[above] {$(3,1)$} (n4);

\node[vertex] (m0) at (0,6.3) {$0$};
\node[vertex] (m1) at (2,6.3) {$0$};
\node[vertex] (m2) at (4,6.3) {$2$};
\node[vertex] (m3) at (6,6.3) {$2$};
\node[vertex] (m4) at (8,6.3) {$0$};
\draw[edge] (m0) edge node[above] {$(3,1)$} (m1);
\draw[edge] (m1) edge node[above] {$(1,3)$} (m2);
\draw[edge] (m2) edge node[above] {$(0,2)$} (m3);
\draw[edge] (m3) edge node[above] {$(2,0)$} (m4);
\end{tikzpicture}
\captionof{figure}{Two $L$-walks whose input-sequence elements are the mother-graph union $C_5 \cup C_6.$}
\label{3_4_L_walks}
\end{center}
The $L$-walks displayed in Figure \ref{3_4_L_walks} correspond to two $(3,4)$-permutiple strings, $(1,3)(0,2)(2,0)(3,1)$ and $(1,3)(0,2)(2,0)(3,1),$ which, in turn, correspond to the two $(3,4)$-permutiple examples $(3,2,0,1)_4=3 \cdot (1,0,2,3)_4$ and $(2,0,1,3)_4=3 \cdot (0,2,3,1)_4.$ We notice that certain cyclic permutations of the inputs in a permutiple string result in another permutiple string. This will happen so long as the first input transitions from the zero state and the last input transitions to the zero state. The reader will also notice that, unlike the previous example, the trivial cycle, $C_0,$ is the only individual cycle  whose image is strongly connected and contains the zero state. 
\end{example}   

Clearly, the permutiples in a class, $C,$ can be generated from the subgraph of $\Gamma$ formed by the union of the images of the cycles of $G_C.$ The edge labels of this subgraph are the cycles of $G_C.$ Our next example highlights this fact.

\begin{example}
\label{4_10}
We will now describe the permutiple class, $C,$ of a familiar example, $p=(8,7,9,1,2)_{10}=4\cdot(2,1,9,7,8)_{10},$ as a subcollection of all $(4,10)$-permutiples.  The $(4,10)$-mother graph, $M,$ is seen in Figure \ref{4_10_mg}, with the edges of $G_C=G_p$ highlighted in red. The cycles of $G_C$ are $C_0=\{(9,9)\},$ $C_1=\{(2,8),(2,8)\},$ and $C_2=\{(1,7),(7,1)\}.$ The $(4,10)$-Hoey-Sloane graph, $\Gamma,$ is seen in Figure \ref{4_10_hsg} with the cycle-image union $\Gamma_0 \cup \Gamma_1 \cup \Gamma_2$ also highlighted in red. We may describe $C$ as all permutiples which are generated from a permutiple string formed from a multiset union involving at least one copy of $\{(2,8),(2,8),(1,7),(7,1)\}.$ 
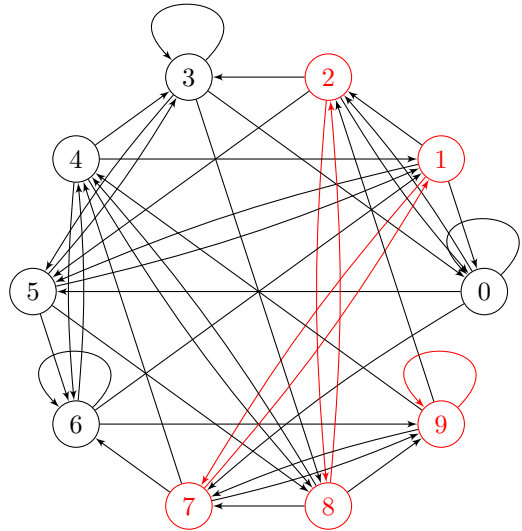
\begin{center}
\begin{tikzpicture}
\tikzset{edge/.style = {->,> = latex'}}
\tikzset{vertex/.style = {shape=circle,draw,minimum size=1.5em}}
[xscale=2, yscale=2, auto=left,every node/.style={circle,fill=blue!20}]
\node[vertex] (n0) at (13,5) {$0$};
\node[vertex,red] (n1) at (12.427051918969068,6.763354468797628) {$1$};
\node[vertex,red] (n2) at (10.927054011580958,7.853168564878643) {$2$};
\node[vertex] (n3) at (9.07295355956129,7.853171024889661) {$3$};
\node[vertex] (n4) at (7.5729527602588975,6.76336090919162) {$4$};
\node[vertex] (n5) at (7.000000000010562,5.00000796076938) {$5$};
\node[vertex] (n6) at (7.572943401820056,3.236651971608781) {$6$};
\node[vertex,red] (n7) at (9.072938417283323,2.1468338951524664) {$7$};
\node[vertex,red] (n8) at (10.927038869289936,2.1468265151194124) {$8$};
\node[vertex,red] (n9) at (12.427042560496046,3.236632650426805) {$9$};
\draw[edge](n0) to[loop] (n0);
\draw[edge, bend right=5](n0) to (n2);
\draw[edge, bend right=0](n0) to (n5);
\draw[edge, bend right=5](n0) to (n7);
\draw[edge, bend right=0](n1) to (n0);
\draw[edge, bend right=0](n1) to (n2);
\draw[edge, bend right=5](n1) to (n5);
\draw[edge, red, bend right=5](n1) to (n7);
\draw[edge, bend right=5](n2) to (n0);
\draw[edge, bend right=0](n2) to (n3);
\draw[edge, bend right=0](n2) to (n5);
\draw[edge, red, bend right=5](n2) to (n8);
\draw[edge, bend right=0](n3) to (n0);
\draw[edge](n3) to[loop] (n3);
\draw[edge, bend right=5](n3) to (n5);
\draw[edge, bend right=0](n3) to (n8);
\draw[edge, bend right=0](n4) to (n1);
\draw[edge, bend right=0](n4) to (n3);
\draw[edge, bend right=5](n4) to (n6);
\draw[edge, bend right=5](n4) to (n8);
\draw[edge, bend right=5](n5) to (n1);
\draw[edge, bend right=5](n5) to (n3);
\draw[edge, bend right=0](n5) to (n6);
\draw[edge, bend right=0](n5) to (n8);
\draw[edge, bend right=0](n6) to (n1);
\draw[edge, bend right=5](n6) to (n4);
\draw[edge](n6) to[loop] (n6);
\draw[edge, bend right=0](n6) to (n9);
\draw[edge, red, bend right=5](n7) to (n1);
\draw[edge, bend right=0](n7) to (n4);
\draw[edge, bend right=0](n7) to (n6);
\draw[edge, bend right=5](n7) to (n9);
\draw[edge, red, bend right=5](n8) to (n2);
\draw[edge, bend right=5](n8) to (n4);
\draw[edge, bend right=0](n8) to (n7);
\draw[edge, bend right=0](n8) to (n9);
\draw[edge, bend right=0](n9) to (n2);
\draw[edge, bend right=0](n9) to (n4);
\draw[edge, bend right=5](n9) to (n7);
\draw[edge,red](n9) to[loop] (n9);
\end{tikzpicture}
\captionof{figure}{The $(4,10)$-mother graph.}
\label{4_10_mg}
\end{center}

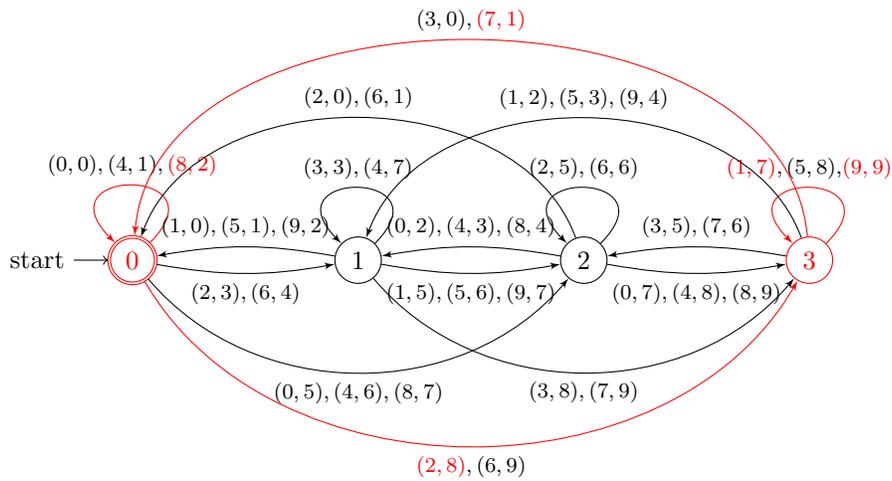
\begin{center}
\begin{tikzpicture}
\tikzset{edge/.style = {->,> = latex'}}
\tikzset{vertex/.style = {shape=circle,draw,minimum size=1.5em}}
[xscale=2, yscale=2, auto=left,every node/.style={circle,fill=blue!20}]
\node[vertex,red,initial,accepting] (n0) at (0,5) {$0$};
\node[vertex] (n1) at (3.0,5) {$1$};
\node[vertex] (n2) at (6.0,5) {$2$};
\node[vertex,red] (n3) at (9.0,5) {$3$};
\draw[edge,red](n0) to[loop] node[above] {\footnotesize{$\color{black}{(0,0), (4,1), }\,\color{red}{(8,2)}$}} (n0);
\draw[edge](n1) to[loop] node[above] {\footnotesize{$(3,3),(4,7)$}} (n1);
\draw[edge](n2) to[loop] node[above] {\footnotesize{$(2,5),(6,6)$}} (n2);
\draw[edge,red](n3) to[loop] node[above] {\footnotesize{$\color{red}{(1,7)}\color{black}{,(5,8),}\color{red}{(9,9)}$}} (n3);
\draw[edge, bend right=10](n0) edge node[below] {\footnotesize{$(2,3),(6,4)$}} (n1);
\draw[edge, bend right=10](n1) edge node[above] {\footnotesize{$(1,0),(5,1),(9,2)$}} (n0);
\draw[edge, bend right=10](n2) edge node[above] {\footnotesize{$(0,2),(4,3),(8,4)$}} (n1);
\draw[edge, bend right=10](n1) edge node[below] {\footnotesize{$(1,5),(5,6),(9,7)$}} (n2);
\draw[edge, bend right=50](n1) edge node[below] {\footnotesize{$(3,8),(7,9)$}} (n3);
\draw[edge, bend right=10](n2) edge node[below] {\footnotesize{$(0,7),(4,8),(8,9)$}} (n3);
\draw[edge, bend right=10](n3) edge node[above] {\footnotesize{$(3,5),(7,6)$}} (n2);
\draw[edge, bend right=70](n3) edge node[above] {\footnotesize{$(1,2),(5,3),(9,4)$}} (n1);

\draw[edge, bend right=50](n0) edge node[below] {\footnotesize{$(0,5),(4,6),(8,7)$}} (n2);
\draw[edge, red, bend right=60](n0) edge node[below] {\footnotesize{$(2,8)\color{black}{,(6,9)}$}} (n3);
\draw[edge, bend right=70](n2) edge node[above] {\footnotesize{$(2,0),(6,1)$}} (n0);
\draw[edge, red, bend right=85](n3) edge node[above] {\footnotesize{$\color{black}{(3,0),}\,\color{red}{(7,1)}$}} (n0);
\end{tikzpicture}
\captionof{figure}{The $(4,10)$-Hoey-Sloane graph.}
\label{4_10_hsg}
\end{center}
Here, we also note that Theorem \ref{conj_class}, when viewed through the lens of the present discussion, tells us that permutiple conjugates arise by either cyclically permuting the sequence of states and inputs simultaneously, or by transposing inputs in a way which preserves the sequence of states. Any combination of the above actions which produces another $L$-walk will result in a conjugate permutiple string. For example, consider the $L$-walk for $p$ shown in Figure \ref{4_10_L_walk_1}.

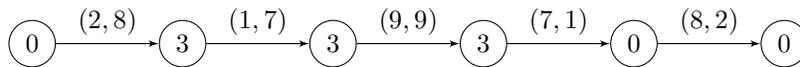
\begin{center}
\begin{tikzpicture}
\tikzset{edge/.style = {->,> = latex'}}
\tikzset{vertex/.style = {shape=circle,draw,minimum size=1.5em}}
[xscale=2, yscale=2, auto=left,every node/.style={circle,fill=blue!20}]
\node[vertex] (n0) at (0,5) {$0$};
\node[vertex] (n1) at (2,5) {$3$};
\node[vertex] (n2) at (4,5) {$3$};
\node[vertex] (n3) at (6,5) {$3$};
\node[vertex] (n4) at (8,5) {$0$};
\node[vertex] (n5) at (10,5) {$0$};
\draw[edge] (n0) edge node[above] {$(2,8)$} (n1);
\draw[edge] (n1) edge node[above] {$(1,7)$} (n2);
\draw[edge] (n2) edge node[above] {$(9,9)$} (n3);
\draw[edge] (n3) edge node[above] {$(7,1)$} (n4);
\draw[edge] (n4) edge node[above] {$(8,2)$} (n5);
\end{tikzpicture}
\captionof{figure}{The $L$-walk corresponding to $p=(8,7,9,1,2)_{10}=4\cdot(2,1,9,7,8)_{10}.$}
\label{4_10_L_walk_1}
\end{center}
 One of the conjugates of $p$ is $(7,1,9,2,8)_{10}=4 \cdot (1,7,9,8,2)_{10},$ and its $L$-walk, shown in Figure \ref{4_10_L_walk_2}, is the result of both a cyclic permutation of the state-input sequence and a transposition of the inputs $(1,7)$ and $(9,9)$ on the $L$-walk of $p.$
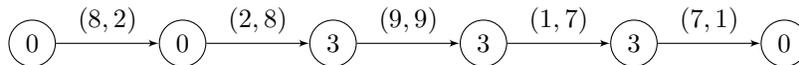
\begin{center}
\begin{tikzpicture}
\tikzset{edge/.style = {->,> = latex'}}
\tikzset{vertex/.style = {shape=circle,draw,minimum size=1.5em}}
[xscale=2, yscale=2, auto=left,every node/.style={circle,fill=blue!20}]
\node[vertex] (n0) at (0,5) {$0$};
\node[vertex] (n1) at (2,5) {$0$};
\node[vertex] (n2) at (4,5) {$3$};
\node[vertex] (n3) at (6,5) {$3$};
\node[vertex] (n4) at (8,5) {$3$};
\node[vertex] (n5) at (10,5) {$0$};
\draw[edge] (n0) edge node[above] {$(8,2)$} (n1);
\draw[edge] (n1) edge node[above] {$(2,8)$} (n2);
\draw[edge] (n2) edge node[above] {$(9,9)$} (n3);
\draw[edge] (n3) edge node[above] {$(1,7)$} (n4);
\draw[edge] (n4) edge node[above] {$(7,1)$} (n5);
\end{tikzpicture}
\captionof{figure}{The $L$-walk corresponding to $(7,1,9,2,8)_{10}=4 \cdot (1,7,9,8,2)_{10}.$}
\label{4_10_L_walk_2}
\end{center}
\end{example} 
 
\begin{remark}\label{broad_class}
For an $(n,b)$-permutiple class, $C,$ mapping the cycles of $G_C$ to their cycle-image union, as seen in the above example, gives us options for classifying permutiples by graph isomorphism (including edge-label collections), which are independent of base, multiplier, and length. 
\end{remark}

Other examples of $(4,10)$-permutiples not considered in the above example, but worth mentioning, are $(7,8,9,1,2)_{10}=4\cdot (1,9,7,2,8)_{10},$ corresponding to the input string $(2,8)(1,2)(9,7)(8,9)(7,1),$ produced by the results of \cite{holt_4}, and an example from the results of \cite{qu}, $(4,9,3,8,2,7,1,5,6)_{10}=4\cdot(1,2,3,4,5,6,7,8,9)_{10},$ with input string $(6,9)(5,8)(1,7)(7,6)(2,5)(8,4)(3,3)(9,2)(4,1).$

As for describing all $(4,10)$-permutiples using the above techniques, we note that, using software, there are 986 cycles on the $(4,10)$-mother graph. This means that the type of analysis seen in the previous examples is hardly feasible. Without further results, we would need to utilize software to continue our exploration. That said, it is still not too difficult to use the Hoey-Sloane graph to find new examples by inspection. For instance, the example $(8,7,1,5,2,8,2,0)_{10}=4\cdot (2,1,7,8,8,2,0,5)_{10}$ was found from the permutiple string $(0,5)(2,0)(8,2)(2,8)(5,8)(1,7)(7,1)(8,2),$ which was obtained by simply examining $\Gamma.$

\subsection{Application to Palintiple Numbers}

In \cite{holt_3}, we find the remark, ``studying the general [permutiple] problem may very well offer insight into particular problems which study only one kind of permutation.'' It is worth mentioning that the methods presented here do indeed bring a novel and unifying perspective to more specific cases of digit permutation problems, namely, palintiple numbers.  

To find all $(n,b)$-palintiples, we may perform the same analysis as above, except that we examine only unions of 1- and 2-cycles of $M$ from which we may form elements of $L.$ It also stands out that for any input string which results in a palintiple, $(d_{0},d_{k})(d_{1},d_{k-1})\cdots(d_{k-1},d_{1})(d_{k},d_{0}),$ the digits, read from left to right, form a palindromic sequence, $d_{0}, d_{k}, d_{1}, d_{k-1}, \ldots,  d_{k-1}, d_{1}, d_{k}, d_{0}.$ 

\begin{example}\label{two_five}
As a simple example, we now describe all $(2,5)$-palintiples. The $(2,5)$-mother graph and the $(2,5)$-Hoey-Sloane graph are seen in Figures \ref{2_5_mg} and \ref{2_5_hsg}.
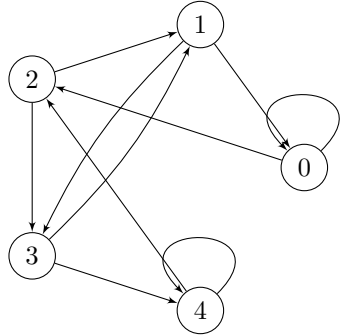
\begin{center}
 \begin{tikzpicture} \tikzset{edge/.style = {->,> = latex'}} \tikzset{vertex/.style = {shape=circle,draw,minimum size=1.5em}} [xscale=2, yscale=2, auto=left,every node/.style={circle,fill=blue!20}] \node[vertex] (n0) at (12,5) {$0$}; \node[vertex] (n1) at (10.618036007720638,6.9021123765857615) {$1$}; 
 \node[vertex] (n2) at (8.381968506839264,6.17557393946108) {$2$}; \node[vertex] (n3) at (8.381962267880038,3.8244346477391877) {$3$}; \node[vertex] (n4) at (10.618025912859958,3.0978843434129413) {$4$}; \draw[edge](n0) to[loop] (n0); 
 \draw[edge](n0) to (n2); 
 \draw[edge](n1) to (n0); 
 \draw[edge, bend right=10](n1) to (n3); 
 \draw[edge](n2) to (n1); 
 \draw[edge](n2) to (n3); 
 \draw[edge, bend right=10](n3) to (n1); 
 \draw[edge](n3) to (n4); 
 \draw[edge](n4) to (n2); 
 \draw[edge](n4) to[loop] (n4); 
 \end{tikzpicture} 
 \captionof{figure}{The $(2,5)$-mother graph.}
 \label{2_5_mg}
\end{center}

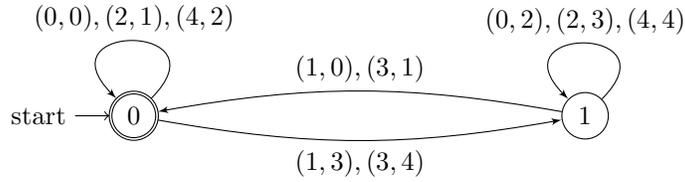
\begin{center}
\begin{tikzpicture}
\tikzset{edge/.style = {->,> = latex'}}
\tikzset{vertex/.style = {shape=circle,draw,minimum size=1.5em}}    
[xscale=2, yscale=2, auto=left,every node/.style={circle,fill=blue!20}]
\node[vertex,accepting,initial] (n0) at (7,5) {$0$};
\node[vertex] (n1) at (13,5) {$1$};
\draw[edge, bend right=10] (n0) edge node[below] {$(1,3),(3,4)$} (n1);
\draw[edge, bend right=10] (n1) edge node[above] {$(1,0),(3,1)$} (n0);
\draw[edge](n0) to[loop] node[above] {$(0,0),(2,1),(4,2)$} (n0);
\draw[edge](n1) to[loop] node[above] {$(0,2),(2,3),(4,4)$} (n1);
\end{tikzpicture}
 \captionof{figure}{The $(2,5)$-Hoey-Sloane graph.}
 \label{2_5_hsg}
\end{center}
The collection of 1- and 2-cycles are $C_0=(0)=\{(0,0)\},$ $C_1=(4)=\{(4,4)\},$ and $C_2=(1,3)=\{(1,3),(3,1)\},$ and their corresponding cycle images are displayed in Table \ref{2_5_ci}.
\begin{center}
\begin{scriptsize}
\begin{tabular}{|c|c|l|c|}
\hline
 & {\bf Mother-Graph Cycle} & {\bf Cycle Images}&\\\hline

$C_0$ &
\begin{tikzpicture} \tikzset{edge/.style = {->,> = latex'}} \tikzset{vertex/.style = {shape=circle,draw,minimum size=1.5em}} [xscale=2, yscale=2, auto=left,every node/.style={circle,fill=blue!20}] 
\node[vertex] (n0) at (4,5) {$0$}; 
\draw[edge](n0) to[loop] (n0); 
\end{tikzpicture} 
&
\begin{tikzpicture} \tikzset{edge/.style = {->,> = latex'}} \tikzset{vertex/.style = {shape=circle,draw,minimum size=1.5em}} [xscale=2, yscale=2, auto=left,every node/.style={circle,fill=blue!20}] 
\node[vertex,accepting,initial] (n0) at (11,5.5) {$0$};
\node[vertex,white] (n1) at (14,5.5) {$1$};
\draw[edge](n0) to[loop] node[pos=0.75,left] {$(0,0)$} (n0);
\end{tikzpicture} 
&
$\Gamma_0$
\\\hline

$C_1$ &
\begin{tikzpicture} \tikzset{edge/.style = {->,> = latex'}} \tikzset{vertex/.style = {shape=circle,draw,minimum size=1.5em}} [xscale=2, yscale=2, auto=left,every node/.style={circle,fill=blue!20}]  
\node[vertex] (n3) at (5.99999601961531,4.000000000007922) {$4$};
\draw[edge](n3) to[loop] (n3);
\end{tikzpicture} 
&
\color{white}
\begin{tikzpicture} \tikzset{edge/.style = {->,> = latex'}} \tikzset{vertex/.style = {shape=circle,draw,minimum size=1.5em}} [xscale=2, yscale=2, auto=left,every node/.style={circle,fill=blue!20}] 
\node[vertex,accepting,initial] (n0) at (11,5.5) {$0$};
\node[vertex,black] (n1) at (14,5.5) {$1$};
\draw[edge,black](n1) to[loop] node[pos=0.25,right] {$(4,4)$} (n1);
\end{tikzpicture} 
&
$\Gamma_1$
\\\hline

$C_2$ &
\begin{tikzpicture} \tikzset{edge/.style = {->,> = latex'}} \tikzset{vertex/.style = {shape=circle,draw,minimum size=1.5em}} [xscale=2, yscale=2, auto=left,every node/.style={circle,fill=blue!20}] 
\node[vertex] (n1) at (6.000001326794896,5.99999999999912) {$1$}; 
\node[vertex] (n2) at (5.00000000000352,5.000002653589793) {$3$}; 
\draw[edge, bend right=10](n1) to (n2); 
\draw[edge,bend right=10](n2) to (n1); 
\end{tikzpicture} 

&

\begin{tikzpicture} \tikzset{edge/.style = {->,> = latex'}} \tikzset{vertex/.style = {shape=circle,draw,minimum size=1.5em}} [xscale=2, yscale=2, auto=left,every node/.style={circle,fill=blue!20}] 
\node[vertex,accepting,initial] (n0) at (11,5.5) {$0$};
\node[vertex] (n1) at (14,5.5) {$1$};
\draw[edge, bend right=10] (n0) edge node[below] {$(1,3)$} (n1);
\draw[edge, bend right=10] (n1) edge node[above] {$(3,1)$} (n0);
\end{tikzpicture} 
&
$\Gamma_2$
\\\hline
\end{tabular}
\captionof{table}{The 1- and 2-cycles of the $(2,5)$-mother graph and their corresponding cycle images.}
\label{2_5_ci}
\end{scriptsize}
\end{center}
From Table \ref{2_5_ci}, we see that any $(2,5)$-palintiple string must involve at least one copy of the multiset $\{(1,3),(3,1)\},$ and possibly copies of $\{(4,4)\}$ or $\{(0,0)\}.$
We may now generate $(2,5)$-palintiple examples. Some examples without leading or trailing zeros include
\[
 \begin{array}{l}
(3,1)_5=2\cdot(1,3)_5,\\
(3,1,3,1)_5=2\cdot(1,3,1,3)_5,\\
(3,4,1)_5=2\cdot(1,4,3)_5,\\
(3,4,1,3,4,1)_5=2\cdot(1,4,3,1,4,3)_5,\\
(3,4,1,0,3,4,1)_5=2\cdot(1,4,3,0,1,4,3)_5,\\
(3,1,3,4,1)_5=2\cdot(1,3,1,4,3)_5,\\
\end{array}
\]
and so forth.
\end{example}

\begin{example}
We may find all $(4,10)$-palintiples of any possible length from the 1- and 2-cycles of the $(4,10)$-mother graph:
\[
\begin{array}{ll}
 \mbox{1-cycles:} & (0), (3), (6), (9)\\ 
 \mbox{2-cycles:} &  (0,2), (1,5), (1,7), (2,8), (3,5), (4,6), (4,8), (7,9).
\\ 
 \end{array}
\]
Examining the images of the above cycles in the same manner as in Examples \ref{two_four}, \ref{three_four}, and \ref{two_five}, any multiset union of cycles must involve at least one multiset copy of $\{(2,8),(8,2),(1,7),(7,1)\},$ and possibly copies of $\{(9,9)\}$ or $\{(0,0)\}.$
\end{example}

\begin{remark}
An equivalent approach to the above methods is to directly keep track of the multiset unions of mother-graph cycles by using the cycle images themselves, defining them as multigraphs; instead of assigning a collection of inputs to each edge, we may define a distinct edge, $(c_1,c_2),$ for each individual input, $(d_1,d_2),$ which satisfies Equation \ref{transition}. Permutiple strings may then be found from any strongly-connected multigraph union which contains the zero state.
\end{remark}

\section{Summary, Conclusions, and Future Lines of Inquiry}

The results and methods developed here produce all permutiples of a given base and multiplier without need of any prior knowledge, such as a known example. The first method uses the $(n,b)$-mother graph, $M,$ to find cycle combinations whose lengths add to a desired number of digits. These cycle combinations give candidates for collections of digits which can be narrowed down by Equation (\ref{digits_carries_2}). The methods of \cite{holt_4} may then be applied to any remaining digit collections. The second method compares the cycles of $M$ to corresponding cycle images, which are subgraphs of the $(n,b)$-Hoey-Sloane graph, $\Gamma.$ Cycle-image unions which are strongly connected and contain the zero state (allowing for $L$-walks) correspond to multiset unions of cycles of $M.$ These multisets may then be ordered into permutiple strings by using the cycle-image union, which is a subgraph of $\Gamma.$ While the first method completes the work of \cite{holt_3,holt_4}, the second method involves substantially less work. It also provides a broader perspective on notions encountered in previous efforts, such as permutiple conjugacy and palintiple numbers. Furthermore, the second method touches upon a line of questioning brought up by \cite{holt_3} involving the prospect of finding a generalization of Young graphs \cite{sloane} which would describe permutiples. Although the second method is not an obvious generalization of Young graphs, it certainly springs from the same soil as previous efforts for describing palintiple numbers \cite{hoey,kendrick_1,sloane,young_1,young_2}. 

Larger bases and multipliers present combinatorial challenges for both methods since the number of cycles of the mother graph grows rapidly with both the base and multiplier. In the case of $(4,10)$-permutiples, we would need either to develop more powerful methods, or develop software to find all possible strongly-connected unions containing the zero state from the 986 cycle images. Generally speaking, it is clear that there are plenty of ideas presented here which can be more fully developed.

Although we could describe the regular language $L$ by resorting to regular expressions, doing so for permutiple strings would be less straightforward since there are constraints on the multiplicities of certain inputs. Using standard regular-expression operations (Kleene star, alternation, etc.) would require that we write a separate regular expression for each strongly-connected cycle-image union containing the zero state, all while taking care to maintain the correct multiplicities of inputs. Moreover, for larger bases and multipliers (more states on $\Gamma$), the task of writing out a regular expression for $L$ quickly becomes onerous. The language which the $(4,10)$-Hoey-Sloane machine recognizes\footnote{``...for which I invite masochists to write the regular expression.''\cite{hoey}} is a case in point. Thus, ``multiset unions of mother-graph cycles which are ordered to form $L$-walks on the Hoey-Sloane graph'' is the simplest description of permutiple strings we have been able to produce. However, knowing that the collection of permutiple strings is closed under concatenation, it would not be surprising if there is a more elegant and compact description of them. 

Finally, the methods presented here may help us to more fully understand the phenomenon of ``derived permutiples'' mentioned in the concluding remarks of \cite{holt_3,holt_4}. We say that a base-$b$ permutiple, $(d_k,\cdots,d_0)_b,$ is {\it derived} if its carries, excluding $c_0,$ form a number, $(c_k,c_{k-1},\ldots,c_1)_n,$ which is a base-$n$ permutiple. As an example, the $(6,12)$-permutiple $(10,3,5,1,8,6)_{12}=6\cdot(1,8,6,10,3,5)_{12}$ has a carry sequence which is the digit sequence of a $(2,6)$-palintiple, $(4,3,5,1,2)_6=2 \cdot (2,1,5,3,4)_6.$ We suspect that broader classification schemes touched upon in Remark \ref{broad_class} may help us to better understand these connections. We leave this, as well as unresolved matters mentioned above, to future efforts.

\vskip 20pt\noindent {\bf Acknowledgements.} The author gratefully extends his thanks to the anonymous referee whose suggestions improved the flow and exposition of this work. The author also thanks his wife for her love and support during the writing of this paper.

\end{document}